\documentclass{amsart}

\input xy
\xyoption{all}

\usepackage{tikz-cd}
\usepackage{geometry}
\usepackage{amsmath}
\usepackage{amssymb}
\usepackage{amscd}
\usepackage{cite}
\usepackage{amsthm}  
\usepackage{tikz}
\usetikzlibrary{er,positioning}

\newtheorem{same}{This should never appear}[section]
\newtheorem{defin}[same]{Definition}

\newtheorem{remark}[same]{Remark}
\newtheorem{theorem}[same]{Theorem}
\newtheorem{example}[same]{Example}
\newtheorem{lemma}[same]{Lemma}
\newtheorem{fact}[same]{Fact}
\newtheorem{question}[same]{Question}
\newtheorem{cor}[same]{Corollary}
\newtheorem{prop}[same]{Proposition}

\newtheorem{hypothesis}[same]{Hypothesis}
\newtheorem{nota}[same]{Notation}

\newtheorem{defin*}{Definition}
\newtheorem*{theorem*}{Theorem}

\makeatletter
\newcommand{\skipitems}[1]{%
  \addtocounter{\@enumctr}{#1}%
}

\newcommand{\rest}{\mathord{\upharpoonright}}
\newcommand{\id}{\textrm{id}}

\newcommand{\A}{\mathbf{A}}
\newcommand{\D}{\mathbf{D}}

\newcommand{\M}{\mathcal{M}}
\newcommand{\mq}{\mathcal{M}_\mathcal{Q}}

\newcommand{\md}{\mathcal{D}}

\newcommand{\LS}{\operatorname{LS}}

\newcommand{\rmod}[1]{\mbox{\rm{Mod}--}{#1}}
\newcommand{\Hom}[3]{\operatorname{Hom}_{#1}(#2,#3)}
\newcommand{\Ext}[4]{\operatorname{Ext}^{#1}_{#2}(#3,#4)}

\newcommand{\gtp}{\mathbf{gtp}}
\newcommand{\gS}{\mathbf{gS}}

\DeclareMathOperator{\im}{im}    

\newbox\noforkbox \newdimen\forklinewidth
\setbox0\hbox{$\textstyle\smile$}
\setbox1\hbox to \wd0{\hfil\vrule width \forklinewidth depth-2pt
 height 10pt \hfil}
\setbox\noforkbox\hbox{\lower 2pt\box1\lower 2pt\box0\relax}
\def\unionstick{\mathop{\copy\noforkbox}\limits}

\def\nonfork_#1{\unionstick_{\textstyle #1}}

\setbox0\hbox{$\textstyle\smile$}
\setbox1\hbox to \wd0{\hfil{\sl /\/}\hfil}
\setbox2\hbox to \wd0{\hfil\vrule height 10pt depth -2pt width
              \forklinewidth\hfil}
\newbox\doesforkbox
\setbox\doesforkbox\hbox{\lower 2pt\box1 \lower 2pt\box2\lower2pt\box0\relax}
\def\nunionstick{\mathop{\copy\doesforkbox}\limits}

\def\fork_#1{\nunionstick_{\textstyle #1}}

\newcommand{\dnf}{\unionstick}

\title{Deconstructible classes of modules and stability}

\author{Marcos Mazari-Armida}
\email{marcos\_mazari@baylor.edu}
\urladdr{https://sites.baylor.edu/marcos\_mazari/}
\address{Department of Mathematics \\ Baylor University \\ Waco, Texas, USA}
\thanks{The first author was partially supported by an NSF grant DMS-2348881, a Simons Foundation grant MPS-TSM-00007597 and an AMS-Simons Travel Grant 2022--2024.}

\author{Jan Trlifaj}
\email{trlifaj@karlin.mff.cuni.cz}
\urladdr{https://www2.karlin.mff.cuni.cz/~trlifaj/en/index.html}
\address{Univerzita Karlova \\ Matematicko-fyzik\'{a}ln\'{\i} fakulta, Katedra algebry \\ Praha, \v{C}esk\'{a} republika}
\thanks{The second author was supported by GA\v CR 23-05148S}

\date{\today.}

\begin{document}

\begin{abstract}
We show that every deconstructible class of modules with all embeddings, all pure embeddings and all RD-embeddings is stable. The argument is presented in the context of abstract classes of modules without amalgamation and the key idea is to construct a \emph{stable-like} independence relation. 

In particular, the following classes of modules with all embeddings, all pure embeddings and all RD-embeddings are shown to be stable: all free and  torsion-free modules over any ring, and for each $n \geq 0$, the classes of all modules of projective and flat dimension $\leq n$ over any ring, and the class of all modules of injective dimension $\leq n$ over any right noetherian ring. 
\end{abstract}

\maketitle

{\let\thefootnote\relax\footnote{{AMS 2020 Subject Classification:
Primary: 03C48, 16E30.  Secondary: 03C45, 03C60, 16D90, 13L05.
}}}


\section{Introduction}

A classical result of first-order model theory, due to Fisher \cite{fisher} and Baur \cite[Theo 1]{baur}, says that if $T$ is a complete first-order theory of $R$-modules then $T$ is stable.\footnote{In the language of this paper, the class of models of $T$ with pure embeddings is stable.} Recently, it was asked if such a result could be extended beyond first-order model theory. More specifically, if it can be extended to abstract elementary classes (AECs for short):

\begin{question}[{\cite[2.12]{maztor}}]\label{mainq} Let $R$ be an associative unital ring and denote the pure submodule relation by $\leq_p$.
If $(K, \leq_p)$ is an abstract elementary class such that $K \subseteq \rmod R$, is $(K, \leq_p)$ stable? 
\end{question} 

 There are numerous cases where the answer is positive -- see \cite[\S 3.2]{bon25} for a recent summary. In this paper, we provide a positive answer for all deconstructible classes of modules, and show that the answer is positive even for embeddings and RD-embeddings:

\vspace{0.1cm}
 \textbf{Theorem \ref{st}.} \textit{Assume $\md$ is $\kappa^+$-deconstructible for $\kappa \geq \operatorname{card}(R) + \aleph_0$. If $\lambda^{\kappa}=\lambda$, then $\md$ with the class of all embeddings, all pure embeddings, and all RD-embeddings is stable in $\lambda$.}

\vspace{0.1cm}
Deconstructible classes of modules constitute an algebraic framework isolated more than twenty years ago. Their importance stems from the fact that they allow the development of relative homological algebra \cite{enje}, \cite{GT}. Many important classes of modules are deconstructible: all free modules and all torsion-free modules over any ring, and for each $n \geq 0$, the classes of all modules of projective and flat dimension $\leq n$ over any ring, and the class of all modules of injective dimension $\leq n$ over any right noetherian ring.  So in particular, Theorem \ref{st} generalizes: \cite[0.3]{baldwine} which shows stability for torsion-free abelian groups with pure embeddings, \cite[4.3]{lrvcell} which shows stability for flat modules with pure embeddings, and \cite[3.24]{maz2} which shows stability for injective modules with embeddings over noetherian rings. 

Theorem \ref{st} is proved by showing that a certain independence relation coming from pushouts (see Definition \ref{r-free}) is \emph{close} to being a stable independence relation. This is similar to \cite[\S 2]{lrvcell}, \cite[\S 4]{maz1}, \cite[\S 3]{mj}, but a key difference to that work and any of the previous work done with independence relations is that we work in a context without amalgamation. This is necessary as deconstructible classes of modules might not have the amalgamation property, see Examples \ref{illus3} and \ref{illus4}.  We are able to show stability by showing that our relation has existence and uniqueness with respect to some \textit{strong} embeddings and that local character  holds for \textit{strong} embeddings (see Lemma \ref{local}). This last result is obtained by merging tools from both module theory and model theory, more precisely we do a downward L\"owenheim-Skolem type of argument which incorporates  the Hill Lemma.

Another important difference between this work and previous work studying classes of modules from the perspective of model theory is that deconstructible classes of modules with all embeddings, all pure embeddings, and all RD-embeddings do not necessarily constitute abstract elementary classes. We characterize when this happens in Corollary \ref{wAEC} and Corollary \ref{corAEC}. Due to this, we work in the framework of abstract classes (first isolated by Grossberg). This allows us to show stability for some classical classes for modules, such as free and projective modules, and suggests that the question of stability is interesting even beyond AECs. 

The paper also provides an example of a natural abstract class which is categorical in every uncountable cardinal, but which does not have the amalgamation property in any infinite cardinal. Hence Shelah's classical result that amalgamation follows from categoricity under weak GCH for all abstract elementary classes \cite{sh88} cannot be extended to our set up (see Example \ref{illus4} and Remark \ref{moreprop}). 

While we were finishing this paper, Paolini and Shelah posted a paper \cite{PS} that focuses on Question \ref{mainq}. They show that it has a negative answer by providing an example of an AEC of torsion-free abelian  groups with the pure subgroup relation that is not stable, see \cite[1.2]{PS}. Nevertheless, they show that the answer is positive for all AECs that have the amalgamation property and tameness \cite[1.3]{PS}. It is worth pointing out that  the classes we consider in this paper might not have the amalgamation property or even be AECs. Hence our results are complementary to the results of Paolini and Shelah.

As the paper uses deep tools both from module theory and model theory, we make an effort to present the background for both subjects. The paper is divided into two sections. Section 2 presents preliminaries, basic results, and many natural examples which elucidate the set-up of the paper. Section 3 has the main results of the paper and shows that deconstructible classes of modules with amalgamation are tame.

The work for this paper began at \emph{WAECO: A workshop on the occasion of Rami Grossberg’s 70th birthday}. The visit of the second author to Baylor was partially supported by an AMS-Simons Travel Grant 2022--2024 of the first author, and by GA\v{C}R grant no.\ 23-05148S. We would like to thank the AMS, GA\v{C}R, and the Simons Foundation for their support.

\section{Preliminaries and basic results}

\subsection{Main framework} Let $R$ be an (associative unital) ring and $\rmod R$ be the category of all (unitary right) $R$-modules with $R$-homomorphisms. For an infinite cardinal $\kappa$, $\rmod R^{< \kappa}$ and $\rmod R^{\leq \kappa}$ will denote its full subcategories consisting of $< \kappa$-presented and $\leq \kappa$-presented modules, respectively. 

A submodule $N$ of a module $M$ is a \emph{pure submodule}, denoted by $N \leq_p M$, in case any finite system of $R$-linear equations with the right-hand side in $N$ which is solvable in $M$, is also solvable in $N$.  Equivalently, for each finitely presented module $F$ and each $f \in \Hom RF{M/N}$, there exists $g \in \Hom RFM$ such that $f = \pi_N g$ where $\pi_N \in \Hom RM{M/N}$ is the canonical projection. Also, $N$ is pure in $M$, if and only if for each finitely presented left $R$-module $F$, the functor $- \otimes_R F$ is exact on the short exact sequence $0 \to N \to M \to M/N \to 0$, see e.g.\ \cite[2.19]{GT}. In particular, $N$ is pure in $M$ in case $M/N$ is a flat module. Moreover, if $M$ is flat, then $N$ is pure in $M$, if and only if $M/N$ is flat. A \emph{pure embedding} is a monomorphism whose image is a pure submodule of its codomain. 
 
A submodule $N$ of a module $M$ is an \emph{RD-submodule}, denoted by $N \leq_{RD} M$, in case $N \cap Mr = Nr$ for each $r \in R$. Equivalently, for each cyclically presented module $C = R/rR$ and each $f \in \Hom RC{M/N}$, there exists $g \in \Hom RFM$ such that $f = \pi_N g$. Also, $N$ is an RD-submodule of $M$, if and only if for each cyclically presented left $R$-module $C$, the functor $- \otimes_R C$ is exact on the short exact sequence $0 \to N \to M \to M/N \to 0$.  An \emph{RD-embedding} is a monomorphism whose image is an RD-submodule of its codomain.

A module $M$ is \emph{torsion-free} if $m.r \neq 0$ for each $0 \neq m \in M$ and each non-zero divisor $r \in R$. Equivalently, $M$ is torsion-free, if Tor$^R_1(M,R/Rr) = 0$ for each non-zero divisor $r \in R$. In particular, all flat modules are torsion-free. For example, if $R$ is an integral domain and $M/N$ is a torsion-free module, then $N$ is an RD-submodule of $M$. If $M$ is moreover torsion-free, then $N$ is an RD-submodule of $M$, if and only if $M/N$ is torsion-free. Any pure submodule is an RD-submodule, and the two notions coincide when $R$ is a Pr\"{u}fer domain.

If $\delta$ is an ordinal, then an increasing chain of modules $\mathcal T = ( A_i \mid i < \delta )$ is \emph{continuous} provided that $A_i = \bigcup_{j<i} A_j$ for each limit ordinal $i < \delta$. $\mathcal T$ is a continuous increasing chain \emph{in $\mathcal{A}$} if moreover $A_i \in \mathcal{A}$ for every $ i < \delta$.

Given a class of modules $\mathcal{C}$ we say that $M$ is a \emph{$\mathcal{C}$-filtered module}, denoted by $M \in \operatorname{Fil}(\mathcal{C})$, if there is an ordinal $\delta$ and a continuous increasing  chain of modules $\mathcal T = ( M_i \mid i \leq \delta)$ such that $M_0 =0$, $M_\delta =M$ and for every $i < \delta$, $M_{i +1}/ M_i \cong C_i$ for some  $C_i \in \mathcal{C}$. The chain $\mathcal T$ is called a \emph{$\mathcal C$-filtration} of $M$ (witnessing that $M$ is $\mathcal C$-filtered). 

\begin{defin}\label{defd} \rm
Let $\mathcal{D}$ be a class of modules. If $\kappa$ is an infinite cardinal such that $\mathcal{D} = \operatorname{Fil}(\mathcal{D}^{<\kappa})$, where $\mathcal{D}^{<\kappa}$ is the class of all $(<\kappa)$-presented modules in $\mathcal{D}$, then $\mathcal{D}$ is called \emph{$\kappa$-deconstructible}. $\mathcal{D}$ is \emph{deconstructible} in case $\mathcal{D}$ is $\kappa$-deconstructible for some infinite cardinal $\kappa$. 
\end{defin}

Deconstructibility is frequent in module theory: deconstructible classes include the classes of all free modules and all torsion-free modules over any ring \cite[8.1]{GT}, and for each $n \geq 0$, the classes of all modules of projective and flat dimension $\leq n$ over any ring \cite[8.10, 8.3]{GT}, and the class of all modules of injective dimension $\leq n$ over any right noetherian ring \cite[8.12]{GT}. In fact, if $\mathcal D$ is any deconstructible class of modules and $n \geq 0$, then the class of all modules of $\mathcal D$-resolution of dimension $\leq n$ is also deconstructible, cf.\ \cite[Corollary 3.5]{ST}. As we will refer to these classes of modules often, for $n \geq 0$, we will denote by $\mathcal P _n$, $\mathcal F_n$, and $\mathcal I_n$ the class of all modules of projective, flat, and injective dimension $\leq n$, respectively, and $Free$ and $\mathcal T \mathcal F$ will denote the classes of all free and torsion-free modules, respectively.

\medskip
We introduce the main framework of the paper. We will explicitly mention when we assume this hypothesis.

\begin{hypothesis}\label{hyp} Let $\kappa \geq \operatorname{card}(R) + \aleph_0$ and $(\mathcal{D}, \M)$ be a pair such that:
\begin{enumerate}
\item $\mathcal{D}$ is a $\kappa^+$-deconstructible class of modules. 
\item $\M$ is either the class of all embeddings, all RD-embeddings, or all pure embeddings.
\end{enumerate} 
\end{hypothesis}

The following two results  are useful when working with deconstructible classes of modules, cf.\ \cite[Chap.\ 7]{GT}. Recall that a class of modules $\mathcal{A}$ is \emph{closed under extensions} if given a short exact sequence $0 \to A \to B \to C \to 0$, if $A, C \in \mathcal{A}$ then $B \in \mathcal{A}$. 

\begin{fact}\label{s-ext}
Let $\md$ be deconstructible. Then $\md$ is closed under transfinite extensions, i.e., $\md$ coincides with the class of all $\md$-filtered modules. In particular, $\md$ is closed under extensions and arbitrary direct sums in $\rmod R$.  
\end{fact}

\begin{fact}[{The Hill Lemma}]\label{Hill} Assume $\kappa$ is a regular infinite cardinal and $\mathcal{C} \subseteq \rmod R ^{<\kappa}$. Suppose that $M$ is a $\mathcal{C}$-filtered module which is  witnessed by a $\mathcal{C}$-filtration $\mathcal{T}= (M_\alpha \mid \alpha \leq \delta)$. Then there is a family $\mathcal{H}$ of submodules of $M$ such that:
\begin{itemize}
\item[{\rm{(H1)}}] $\mathcal{T} \subseteq \mathcal{H}$. In particular, $0 \in \mathcal{H}$ and $M \in \mathcal{H}$.
\item[{\rm{(H2)}}] $(\mathcal{H}, \subseteq)$ is a complete distributive sublattice of the modular lattice of all submodules of $M$. In particular, $\mathcal{H}$ is closed under arbitrary sums, intersections, and unions of  continuous increasing chains.
\item[{\rm{(H3)}}] If $N, P \in \mathcal{H}$ and $N \subseteq P$, then $P/N$ is $\mathcal{C}$-filtered. In particular, since $0 \in \mathcal H$, all modules $P \in \mathcal{H}$ are $\mathcal{C}$-filtered.
\item[{\rm{(H4)}}] For every $N \in \mathcal{H}$ and $X \subseteq M$, if $|X| < \kappa$, then there is $P \in \mathcal{H}$ such that $N \cup X \subseteq P$ and $P/N$ is $(<\kappa)$-presented. 
\end{itemize}
\end{fact}

Let $\mathcal A$ be a class of modules. We will denote by $\varinjlim \mathcal A$ ( $\varinjlim _\omega \mathcal A$ ) the class of all modules $M$ such that $M$ is the direct limit of a direct system (a countable direct system) consisting of modules from $\mathcal A$. For example, if $\mathcal A = \mathcal P_0$, then $\varinjlim \mathcal A = \mathcal F _0$. A class of modules $\mathcal A$ is \emph{closed under (countable) direct limits} in case  $\mathcal A = \varinjlim \mathcal A$ ($\mathcal A = \varinjlim _\omega \mathcal A$).

\subsection{Abstract elementary classes} A pair $\mathbf{A}= (\mathcal{A}, \lesssim_{\mathbf{A}})$ is an \emph{abstract class of modules} (or an \emph{AC} for short), if  $\mathcal A$ is a class of modules\footnote{In the language of modules  $\{0, +,-\} \cup \{ \cdot r   \mid r \in R \}$ where $R$ is a fixed ring and $\cdot r$ is interpreted as right multiplication by $r$ for each $r \in R$.}, $\lesssim_{\mathbf{A}}$ is a partial order on $\mathcal A$ such that for all $A, B \in \mathcal A$, if $A \lesssim_{\mathbf{A}} B$, then $A$ is a submodule of $B$, and both $\mathcal A$ and $\lesssim_{\mathbf{A}}$ are closed under isomorphisms (in the latter case, this means that $N^\prime \lesssim_{\mathbf{A}} M^\prime$, whenever $N \lesssim_{\mathbf{A}} M$ and there is an isomorphism $f\colon M \to M^\prime$ such that $f(N) = N^\prime$). We denote the cardinality of a module $A$ by $\|A\|$.

So clearly, if $(\mathcal{D}, \M)$ satisfies Hypothesis \ref{hyp}, then $\mathbf{D} = (\mathcal{D}, \leq_\M)$ is an AC where $A \leq_\M B$ if the inclusion from $A$ to $B$ is in $\M$ and $A, B \in \mathcal{D}$.

We will say that $\lesssim_{\mathbf{A}}$ \emph{refines direct summands} provided that $A \lesssim_{\mathbf{A}} B$ whenever $A, B \in \mathcal A$ and $A$ is a direct summand in $B$. This is the case, for example, when $\lesssim_{\mathbf{A}}$ is the partial order induced by $\mathcal M$ when $\mathcal M$ is the class of all inclusions, all pure inclusions, or all RD-inclusions.

\begin{defin}\label{defaec} \rm
An abstract class of modules $\mathbf A = (\mathcal A, \lesssim_{\mathbf{A}} )$ is an \emph{abstract elementary class} (or an \emph{AEC}) of modules, in case the following conditions are satisfied:
\begin{itemize}
\item[{\rm{(A1)}}] \emph{Tarski-Vaught Axioms:} If $( A_i \mid i < \delta )$ is a continuous $\lesssim_{\mathbf{A}}$-increasing chain in $\mathcal A$, then
\begin{enumerate}
\item \emph{Closure under unions of $\lesssim_{\mathbf{A}}$-chains:} $\bigcup_{i<\delta} A_i \in \mathcal A$,
\item \emph{$\lesssim_{\mathbf{A}}$-continuity:} $A_j \lesssim_{\mathbf{A}} \bigcup_{i<\delta}A_i$ for each $j<\delta$, and
\item \emph{Smoothness:} If $M \in \mathcal A$ and $A_i \lesssim_{\mathbf{A}}  M$ for each $i<\delta$, then $\bigcup_{i<\delta}A_i \lesssim_{\mathbf{A}} M$.
\end{enumerate}
\item[{\rm{(A2)}}] \emph{Coherence:} If $A,B,C \in \mathcal A$, $A \lesssim_{\mathbf{A}} C$, $B \lesssim_{\mathbf{A}} C$, and $A$ is a submodule of $B$, then $A \lesssim_{\mathbf{A}} B$.
\item[{\rm{(A3)}}] \emph{L\"{o}wenheim-Skolem-Tarski Axiom:} There exists a cardinal $\lambda$ such that, if $A$ is a submodule of a module $B \in \mathcal A$, then there exists
$A^\prime \in \mathcal A$ such that $A \subseteq A^\prime \lesssim_{\mathbf{A}} B$, and $\|A^\prime \| \leq \|A \| + \lambda$. The least such infinite cardinal $\lambda \geq \operatorname{card}(R)$ is called the \emph{L\" owenheim--Skolem number} of $\mathbf A$, and it is denoted by $\LS(\mathbf A)$.
\end{itemize}
\end{defin}

If $(\mathcal{D}, \M)$ satisfies Hypothesis \ref{hyp}, then clearly coherence holds for $\mathbf{D} = (\mathcal{D}, \leq_\M)$ and if (A1)(1) holds for $\mathbf{D}$ then (A1)(2)(3) hold for $\mathbf{D}$\footnote{Observe that (A1)(2)(3) hold as they hold for $(\rmod R, \leq_\M)$.}.
Moreover, $\D$ satisfies the L\"owenheim-Skolem-Tarski Axiom. The proof of the latter fact follows from $(\rmod R, \leq_\M)$ being an AEC with  L\" owenheim--Skolem number $\operatorname{card} (R) + \aleph_0$, and from the following consequence of the Hill Lemma which will be used a few times:

\begin{prop}\label{small} Assume that $\kappa$ is an infinite cardinal, and $\mathcal{D}$ is a $\kappa^+$-deconstructible class of modules. 
Let $A \in \mathcal{D}$ and $\mathcal{H}$ be the Hill family from Fact \ref{Hill} for the regular infinite cardinal $\kappa^+$, $\mathcal C = \md ^{\leq \kappa}$, and for a $\mathcal C$-filtration $\mathcal T$ of $A$. If $B$ is a submodule of $A$, then there is $C \in \mathcal{H}$ such that $B \subseteq C$ and $\| C \| \leq \| B \| + \kappa$. 
\end{prop} 
\begin{proof} Let  $\mathcal{P}^{<\omega}(B)$ be the set of finite subsets of $B$. For every $X \in \mathcal{P}^{<\omega}(B)$, let $C_X \in \mathcal{H}$ be such that $X \subseteq C_X$ and $\| C_X \| \leq \kappa$. $C_X$ exists by Fact \ref{Hill}.(H4). Let $C = \sum_{X \in \mathcal{P}^{<\omega}(B)} C_X$. $C \in \mathcal{H}$ by Fact \ref{Hill}.(H2) and clearly $B \subseteq C$ and $\| C \| \leq \| B \| + \kappa$.
\end{proof}
 
\begin{lemma}\label{ls-axiom} Assume $(\mathcal{D}, \M)$ satisfies Hypothesis \ref{hyp}. Then $\mathbf{D} = (\mathcal D, \leq_\M)$ satisfies the L\"owen\-heim-Skolem-Tarski Axiom \ref{defaec}.(A3) with $\LS(\D) \leq \kappa$. That is, if $B$ is a submodule of $A \in \mathcal{D}$, then there is $C \in \mathcal{D}$ such that $B \subseteq C \leq_\M A$ and $\| C \| \leq \| B \| + \kappa$. 
\end{lemma}
\begin{proof} Let $B$ be submodule of $A \in \md$.  Let $\lambda = \| B  \| + \kappa$. If $\| A \| \leq \lambda$, we can take $C = A$.   

Assume $\| A \| > \lambda$. Since $A \in \md$, there is a $\md ^{\leq \kappa}$-filtration $\mathcal T$ of $A$. Let $\mathcal H$ be the Hill family from Fact \ref{Hill} (for the regular infinite cardinal $\kappa^+$ and $\mathcal C = \md ^{\leq \kappa}$) obtained from $\mathcal{T}$. Let $A_0 = 0$ and $C_0= B$. 

By induction on $0 < n < \omega$, we will construct a chain $\{ C_n \mid 0 < n < \omega \}$ consisting of $\M$-submodules of $A$ of cardinality $\leq \lambda$, and a chain $\{ A_n \mid 0 < n < \omega \} \subseteq \mathcal H$ consisting of submodules of $A$ of cardinality $\leq \lambda$ such that $C_n \subseteq A_{n+1} \subseteq C_{n+1}$ for all $n < \omega$. Then $C = \bigcup_{n < \omega} C_n = \bigcup_{n < \omega} A_n \in \md$ by Fact \ref{Hill}.(H2), $C$ is a $\M$ -submodule of $A$, and  $\| C \| \leq \lambda$. 

Let $n < \omega$. Since $\| C_n \| \leq \lambda$ it follows from Proposition \ref{small} that there exists $A_{n+1} \in \mathcal H$ such that $C_n \leq A_{n+1}$ and $\| A_{n+1} \| \leq \lambda$. Since $\operatorname{card}(R) + \aleph_0 \leq \lambda$, there exists a $\M$-submodule $C_{n+1}$ of $A$ containing $A_{n+1}$ such that $\| C_{n+1} \| \leq \lambda$.   
\end{proof}

Lemma \ref{ls-axiom} implies that the least cardinal $\kappa \geq \operatorname{card}(R) + \aleph_0$ such that $\mathcal D$ is $\kappa^+$-deconstructible provides an upper bound for the L\"{o}wenheim--Skolem number of $\mathbf{D}$. However, $\LS(\D)$ may be much smaller than this upper bound, as shown by the following example.

\begin{example}\label{restrictedflatML} \rm Let $R$ be a right hereditary non-right perfect ring. Denote by $\mathcal F \mathcal M$ the class of all $\aleph_1$-projective (= flat Mittag-Leffler) modules. Let $\lambda$ be \emph{any} cardinal $\geq  \operatorname{card}(R) (\geq \aleph_0)$. Let $C_\lambda$ be the class of all $\leq \lambda$-presented modules from $\mathcal F \mathcal M$, and $D_\lambda$ the class of all $C_\lambda$-filtered modules. Since $R$ is right hereditary, $C_\lambda$ is closed under submodules, and so is $D_\lambda$. 

By [GoTr12, Theorem 10.13], there exists a $\lambda^+$-presented module $M_{\lambda^+} \in \mathcal F \mathcal M$ such that $M_{\lambda^+} \notin D_\lambda$. Let $\mathcal D$ be the class of all $C_{\lambda^+}$-filtered modules. Then $\mathcal D$ is $\lambda^{++}$-deconstructible, but not $\lambda^+$-deconstructible. So $\kappa = \lambda^+$ is the least cardinal such that $\mathcal D$ with the submodule (pure submodule, or RD-submodule) relation satisfies our Hypothesis \ref{hyp}. However, in either case, $\LS(\D) =  \operatorname{card}(R) < \kappa$, because $\mathcal D$ is closed under submodules. 
\end{example}

\begin{cor}\label{wAEC} Assume $(\mathcal{D}, \M)$ satisfies Hypothesis \ref{hyp}. Then $\mathbf{D} = (\mathcal D, \leq_\M)$ satisfies all the conditions from Definition \ref{defaec} except possibly for the Tarski-Vaught Axioms \ref{defaec}.(A1).

Furthermore, $\mathbf{D}$ is an AEC with $\LS(\D) \leq \kappa$ if and only if $\md$ is closed under unions of $\M$-chains.
\end{cor}

In the case when $\md$ is closed under direct summands, the condition of $\mathbf{D}$ being an AEC can be phrased in a stronger form. In order to see this we recall the following basic property of deconstructible abstract elementary classes of modules from \cite{satr}:

\begin{fact}\label{decaec}
Let $(\mathcal D, \lesssim_{\mathbf{A}})$ be an abstract elementary class of modules such that $\mathcal D$ is deconstructible and closed under direct summands, and $\lesssim_{\mathbf{A}}$ refines direct summands. Then $\mathcal D$ is closed under direct limits. 
\end{fact}

\begin{cor}\label{corAEC}  Assume that $(\mathcal{D}, \M)$ satisfies Hypothesis \ref{hyp} and $\md$ is closed under direct summands. Then $(\md, \leq_\M)$ is an AEC if and only if $\md$ is closed under direct limits in $\rmod R$. 

In particular, whether or not $(\md, \leq_\M)$ is an AEC depends only on the closure properties of $\md$, not on the particular choice of the partial order $\leq _\M$.  
\end{cor}

We now pause to illustrate the variety of settings that arise for $\mathbf{D}$.    

\begin{example}\label{q2} \rm \
\begin{enumerate}
\item For each $n \geq 0$, the classes $\mathcal{F}_n$ (for any ring $R$) and $\mathcal{I}_n$ (for $R$ right noetherian) are deconstructible and closed under direct limits. Hence they satisfy Hypothesis \ref{hyp} and form AECs for any choice of the embeddings $\M$. 
\item Let $R$ be any ring, $\mathcal C$ a class consisting of (some) pure-injective modules and $\mathcal D = {}^{\perp_{\infty}} \mathcal C := \{ M \in \rmod R \mid \Ext iRMC = 0 \hbox { for all } i > 0 \hbox{ and all } C \in \mathcal C \}$. Then the class $\mathcal D$ is deconstructible and closed under direct limits \cite[6.21]{GT}. Hence it satisfies Hypothesis \ref{hyp} and forms an AEC for any choice of the embeddings $\M$. 
\item The class $\mathcal{P}_n$ is deconstructible for any ring $R$ and any $n \geq 0$. Hence it satisfies Hypothesis \ref{hyp} for any choice of the embeddings $\M$. As $\mathcal{P}_n$ is closed under direct summands, in order for $\mathcal{P}_n$ to be an AEC for a choice of $\M$, $\mathcal{P}_n$ has to be closed under direct limits by Corollary \ref{corAEC}. This is quite a strong condition:

For $n = 0$, it just says that $\mathcal{P}_0 = \mathcal{F}_0$, i.e., $R$ is a right perfect ring (see for example \cite[28.4]{AF}), and we are in the setting of part (1) for $n = 0$.

If $R$ is an integral domain then $\varinjlim \mathcal P_1 = \mathcal F_1$, so $\mathcal{P}_1$ is closed under direct limits, iff $\mathcal P_1 = \mathcal F_1$, that is, $R$ is an almost perfect domain (cf.\ \cite[Theorems 9.4 and 9.10]{GT}). So we are again in the setting of part (1), but for $n =1$.
\end{enumerate}
\end{example}

\begin{remark} \rm
Assume $\mathcal A$ is any class of modules and $\M$ is either the class of all embeddings, all RD-embeddings, or all pure embeddings. Assume moreover that $\mathcal A$ is closed under $\M$-submodules. Then $(\mathcal A,\leq_{\M})$ is an AEC, iff $\mathcal A$ is closed under unions of $\M$-chains. This simple fact yields many examples of AECs of modules, some of which do not satisfy Hypothesis \ref{hyp}. 

An interesting example is provided by the class $\mathcal F \mathcal M$ of all $\aleph_1$-projective (= flat Mittag-Leffler) modules with pure embeddings over any non-right perfect ring $R$. Then $\mathbf{FM} = (\mathcal F \mathcal M,\leq_p)$ is an AEC with $\LS(\mathbf{FM}) = \operatorname{card}(R) (\geq \aleph_0)$, because $\mathcal F \mathcal M$ is closed under pure submodules and unions of pure chains (cf.\ \cite[Corollary 3.20(a)]{GT}). However, since the ring $R$ is not right perfect, the class $\mathcal F \mathcal M$ is neither deconstructible nor closed under direct limits (cf.\ \cite[Theorem 10.13]{GT}).   
\end{remark}

\begin{remark} \rm 
By Corollary \ref{wAEC}, closure under unions of $\M$-chains is sufficient for $\mathbf{D} = (\md, \leq_\M)$ satisfying Hypothesis \ref{hyp} to be an AEC. However, the weaker notion of having chain bounds is not sufficient in general. Recall that $(\mathcal D, \lesssim_{\mathbf{A}})$ has \emph{chain bounds}, if for every increasing continuous $\lesssim_{\mathbf{A}}$-chain of modules $\{ M_i \mid i < \delta \} \subseteq \mathcal D$, there is $N \in \mathcal D$ such that $M_i \lesssim_{\mathbf{A}}  N$ for all $i< \delta$. 

For an example, let $R$ be a valuation domain such that $R$ is not noetherian. Denote by $Q$ the quotient field of $R$, and let $\mathcal D$ be the class of all $\{ R, Q \}$-filtered modules (the modules in $\mathcal D$ are called \emph{strongly flat}, cf.\ \cite[Theorem 7.64]{GT}). Clearly, $\mathcal D$ is deconstructible, and it is closed under direct summands by \cite[Corollary 7.51]{GT}. As $\mathcal P_0 \subsetneq \mathcal D \subsetneq \mathcal F _0$ by \cite[Theorem 7.56]{GT}, $\mathcal D$ is not closed under direct limits. Hence $\mathbf{D}$ is not an AEC for any particular choice of the partial order $\leq _\M$ by Corollary \ref{corAEC}. However, $(\mathcal D, \leq)$ has chain bounds, as we can always take $N = E(\bigcup_{i < \delta} M_i)$, the injective hull of the union of a chain $\{ M_i \mid i < \delta \} \subseteq \mathcal D$. Since the union is flat, its injective hull is isomorphic to a direct sum of copies of $Q$, hence $N \in \mathcal D$. 
\end{remark}
\newpage
\subsection{Stability} 
Let $\mathbf{A} = (\mathcal A, \lesssim_{\mathbf{A}})$ be an abstract class of modules. We say that $f: A_0 \to A_1$ is an \emph{$\mathbf{A}$-embedding} if $f$ is a monomorphism, $A_0, A_1 \in \mathcal{A}$  and $f(A_0) \lesssim_{\mathbf{A}} A_1$.  We write $f: A_0 \xrightarrow[B]{} A_1$ if $f$ is an $\mathbf{A}$-embedding, $B$ a submodule of $A_0$, and $f \upharpoonright B = \id_B$.

Types for arbitrary AECs were introduced by Shelah  \cite{sh300}. These extend first-order types to this setting. We introduce types for ACs of modules.

\begin{defin}\label{gtp-def} \rm
  Let $\mathbf{A}= (\mathcal{A}, \lesssim_{\mathbf{A}})$ be an abstract class of modules.
  
  \begin{enumerate}
    \item Let $\A^3$ be the set of triples of the form $(\bar{b}, A,
N)$, where $N \in \mathcal{A}$, $A$ is a submodule of $N$, and $\bar{b}$ is a tuple (possibly infinite) from $N$. 
    \item For $(\bar{b}_1, A_1, N_1), (\bar{b}_2, A_2, N_2) \in \A^3$, we
say that  $(\bar{b}_1, A_1, N_1)$, $(\bar{b}_2, A_2, N_2)$ are \emph{atomically related} (in $\A$), denoted by $(\bar{b}_1, A_1, N_1)E_{\text{at}}^{\A} (\bar{b}_2, A_2, N_2)$, if $A
:= A_1 = A_2$, and there exist $\A$-embeddings $f_i : N_i \xrightarrow[A]{} N$ for $i = 1, 2$ such that  
$f_1 (\bar{b}_1) = f_2 (\bar{b}_2)$ and $N \in \mathcal{A}$.
    \item Note that $E_{\text{at}}^{\A}$ is a symmetric and
reflexive relation on $\A^3$. We let $E^{\A}$ be the transitive
closure of $E_{\text{at}}^{\A}$.
    \item For $(\bar{b}, A, N) \in \A^3$, let $\gtp_{\A} (\bar{b} / A;
N) := [(\bar{b}, A, N)]_{E^{\A}}$, the $E^{\A}$-equivalence class of $(\bar{b}, A, N)$. We call such an equivalence class the
\emph{(Galois) type} of $\bar{b}$ over $A$ in $N$. Usually, $\A$ will be clear from the context and we will omit it.
\item  For $p=\gtp_{\A} (\bar{b} / M; N)$ and $M_0$ a submodule of  $M$, the \emph{restriction} of $p$ to $M_0$ is $\gtp_{\A} (\bar{b} / M; N)\upharpoonright {M_0}:= [(\bar{b}, M_0, N)]_E$
\item For $M \in \mathcal{A}$, $\gS_{\A}(M)= \{  \gtp_{\A}(b / M; N) \mid M \lesssim_{\mathbf{A}} N\in \mathcal{A} \text{ and }  b \text{ is an element of } N\} $.
.
  \end{enumerate}
\end{defin}

 Let $M \in \A$.
Notice that even though there is a proper class of $N$ extending $M$,  $\gS_{\A}(M)$ is a set if $\A$ satisfies the L\"owenheim-Skolem-Tarski Axiom \ref{defaec}.(A3). Moreover, we have the following:

\begin{fact} Assume $\A$ is an abstract class of modules that  satisfies the L\"owenheim-Skolem-Tarski Axiom \ref{defaec}.(A3). If  $M \in \A$ with $\| M \| \geq  LS(\A)$, then $| \gS_{\A}(M) | \leq 2^{\|M\|}$. 
\end{fact}
\begin{proof}
Let $M \in \A$ and  $\lambda := \|M\|$.  Assume for the sake of contradiction that $| \gS_{\A}(M) | > 2^{\lambda}$. Let  $\{ p_i = \gtp(a_i/M;N_i) \mid i < (2^{\lambda})^+ \}$ be an enumeration without repetitions in  $\gS_{\A}(M)$.  We may assume that $\| N_i \| = \lambda$ for every $i < (2^{\lambda})^+$, as otherwise we could apply the  L\"owenheim-Skolem-Tarski Axiom \ref{defaec}.(A3) to $M \cup \{a_i\}$ to get a smaller $N_i$. 

Let $\{ m_\alpha \mid \alpha < \lambda \}$ be  an enumeration of $M$. Let $\{c\} \cup \{ d_\alpha \mid \alpha < \lambda \}$ be new constant symbols not in the language of modules and  $ \Delta = \{ (N, c^N, (d_\alpha^N)_{\alpha < \lambda})/ \cong  \mid N \in \mathcal{A} \text{ with } \|N\| = \lambda, c^N \in N \text{ and } d_\alpha^N \in N \text{ for all } \alpha < \lambda \}$.

Let $\Psi: (2^{\lambda})^+ \to \Delta$ be given by $\Psi(i)= (N_i, a_i, (m_\alpha)_{\alpha < \lambda})/\cong$. Since $\| \Delta \| \leq 2^\lambda$. By the pigeon-hole principle there are $ i \neq j \in (2^{\lambda})^+ $ and  $f: (N_i, a_i, (m_\alpha)_{\alpha < \lambda}) \cong (N_j, a_j, (m_\alpha)_{\alpha < \lambda})$. Observe that $f: N_i \to N_j$ is an isomorphism, $f(a_i) = a_j$ and $f \upharpoonright M = \id_{M}$. Hence $p_i = p_j$. This contradicts the fact that $p_i \neq p_j$.
\end{proof}

%

\medskip 
\begin{remark} \rm As $\mathbf{D} = (\md, \leq_\M)$  satisfies the L\"owenheim-Skolem-Tarski Axiom \ref{defaec}.(A3) when  $(\M, \mathcal{D})$ satisfies Hypothesis \ref{hyp}, $ \gS_{\D}(B)$ is always a set for every $B \in \md$.
\end{remark}

\medskip
The following is the key notion of the paper. For AECs, it was introduced by Shelah.

\begin{defin} \rm
 An AC of modules $\A$ is \emph{stable in $\lambda$}  if for any $M \in
\mathcal{A}$ with $\| M \| = \lambda$, $| \gS_{\A}(M) | \leq \lambda$. An AC of modules is \emph{stable} if there is a $\lambda \geq \LS(\A)$ such that $\A$ is stable in $\lambda$.
\end{defin}

\subsection{Amalgamation}\label{amalg} A pair of monomorphisms $f_1 : A_0 \to A_1$ and $f_2 : A_0 \to A_2$ in $\rmod R$ is called a \emph{span}. Given a span $(f_1,f_2)$, we can form its pushout $(P, h_1,h_2)$ in $\rmod R$ and obtain the following commutative diagram with exact rows and columns

$$\begin{CD}
@.     0@.        0      @.     @.\\
@.     @VVV       @VVV   @.     @.\\
0@>>>  A_{0}@>{f_1}>> A_{1}@>>>  B_{1}@>>>  0\\
@.     @V{f_2}VV  @V{h_1}VV   @|     @.\\
0@>>>  A_{2}@>{h_2}>>      P@>>>  B_{1}@>>>  0\\
@.     @VVV       @V{\pi}VV @.  @.\\
@.     B_{2} @=       B_{2}      @.     @.\\
@.     @VVV       @VVV   @.     @.\\
@.     0@.        0.      @.     @.
\end{CD}$$
  
Recall that $P \cong (A_1 \oplus A_2)/I$ where $I = \{ (f_1(a),-f_2(a)) \mid a \in A_0 \}$, $h_1 (a_1) = (a_1,0) + I$, and $h_2 (a_2) = (0,a_2) + I$. In particular, $P = \im {h_1} + \im {h_2}$. By the characterizations above, it easily follows that if $f_1$ is a pure embedding (RD-embedding), then so is $h_2$, and similarly for $f_2$ and $h_1$. 

Let $(f_1,f_2)$ be a span. An \emph{amalgam} of $(f_1,f_2)$ in $\rmod R$ is a triple $(A_3, g_1, g_2)$ such that $g_1$ and $g_2$ are monomorphisms, and the following diagram is commutative

 \[ \begin{tikzcd}
A_1 \arrow[r, "g_1"]                         & A_3                   \\
A_0 \arrow[u, "f_1"] \arrow[r, "f_2"'] & A_2 \arrow[u, "g_2"']
\end{tikzcd} \]

Any amalgam of the span $(f_1,f_2)$ factors uniquely through the pushout $(P, h_1,h_2)$ of $(f_1,f_2)$ in $\rmod R$, that is, there is a unique homomorphism $t \in \Hom RP{A_3}$ such that $g_1 = t h_1$ and $g_2 = t h_2$. Notice however that even though $g_1$ and $g_2$ are monomorphisms,  $t$ is not a monomorphism in general. 

\begin{example}\label{ex1} \rm
Let $R$ be an integral domain which is not a field, and $Q$ be its quotient field. Let $A_0 = R$, $A_1 = A_2 = A_3 = Q$, $f_1 = f_2 : R \hookrightarrow Q$, and $g_1 = g_2 = id_{Q}$. Then $(Q, g_1, g_2)$ is an amalgam of the span $(f_1,f_2)$, $P \cong (Q \oplus Q)/\{ (r,-r) \mid r \in R \} $, and $t$ is not a monomorphism as $P$ is not torsion-free. 
\end{example}

\medskip
Assume $(A_3, g_1,g_2)$ is an amalgam of a span $(f_1,f_2)$. Then $(A_3,g_1,g_2)$ is a \emph{disjoint amalgam} in case $g_1(A_1) \cap g_2(A_2) = g_1f_1(A_0) (= g_2f_2(A_0))$. We note the following simple characterization:

\begin{lemma}\label{test} Let $(A_3,g_1,g_2)$ be an amalgam of the span $(f_1,f_2)$ and $(P, h_1, h_2)$ be the pushout of $(f_1,f_2)$.

Then $(A_3,g_1,g_2)$ is a disjoint amalgam, iff the unique homomorphism $t$ from the pushout $P$ to $A_3$ is monic. 

In particular, if $(A_3,g_1,g_2)$ is the pushout of $(f_1,f_2)$ in $\rmod R$, then the amalgam $(A_3,g_1,g_2)$ is disjoint. 
\end{lemma}
\begin{proof} Assume the amalgam is disjoint and let $x \in \ker t$. Since $P = \im {h_1} + \im {h_2}$, $x = h_1(a_1) - h_2(a_2)$ for some $a_i \in A_i$ ($i = 1,2$). So $0 = t h_1(a_1) - t h_2(a_2) = g_1(a_1) - g_2(a_2)$. Since the maps $f_i$, $g_i$ ($i = 1, 2$) are monic, the disjointness yields $g_1(a_1) = g_1f_1(a_0)$ and $g_2(a_2) = g_2f_2(a_0)$ for some $a_0 \in A_0$, and $a_1 = f_1(a_0)$ and $a_2 = f_2(a_0)$. Hence $x = h_1f_1(a_0) - h_2f_2(a_0) = 0$.

Conversely, assume $t$ is monic and let $y \in g_1(A_1) \cap g_2(A_2)$. Then $y = t h_1(a_1) = t h_2(a_2)$ for some $a_1 \in A_1$ and $a_2 \in A_2$, whence $h_1(a_1) = h_2(a_2)$. So $(a_1,-a_2) \in I$, whence $a_1 = f_1(a_0)$ and $a_2 = f_2(a_0)$ for some $a_0 \in A_0$. Then $y = g_1(f_1(a_0)) \in g_1f_1(A_0)$.        
\end{proof}

\begin{defin} \rm
Let $\mathbf{A} = (\mathcal A, \lesssim_{\mathbf{A}})$ be an abstract class of modules. 

\begin{enumerate}
\item $\mathbf{A}$ is \emph{closed under pushouts} if for every span $(f_1: A_0 \to A_1, f_2: A_0 \to A_2)$ with $A_0, A_1, A_2 \in \mathcal{A}$ and $f_1, f_2$ $\mathbf{A}$-embeddings, the pushout $(P, g_1, g_2)$  (in $\rmod R$) is an  amalgam of $(f_1, f_2)$ such that $P \in \mathcal{A}$ and $g_1, g_2$ are $\mathbf{A}$-embeddings.

\item Let $\mu \leq \lambda$ be  infinite cardinals. Then $\mathbf{A}$ has the \emph{amalgamation property} in $[\mu, \lambda]$, if for every span $(f_1: A_0 \to A_1, f_2: A_0 \to A_2)$ with $A_i \in \mathcal{A}$ and $\mu \leq \|A_i \| \leq \lambda$ for $i = 0, 1, 2$, and $f_1, f_2$ $\mathbf{A}$-embeddings, there is an  amalgam $(A_3, g_1, g_2)$ of $(f_1, f_2)$ such that $A_3 \in \mathcal{A}$ with $\mu \leq \| A_3 \| \leq \lambda$ and $g_1, g_2$ are $\mathbf{A}$-embeddings. If $\lambda =\mu$, we say that  $\mathbf{A}$ has the \emph{amalgamation property} in $\lambda$.

\item $\mathbf{A}$ has the \emph{amalgamation property}, if $\mathbf{A}$ has the amalgamation property in $[\aleph_0, \lambda]$ for every infinite cardinal $\lambda$.

\end{enumerate}
\end{defin}

It is clear that if $\mathbf{A}$ is closed under pushouts, then $\mathbf{A}$ has the amalgamation property. Below we provide examples where the converse fails. We will be mostly interested in the amalgamation property.

We focus on the pairs $(\mathcal{D}, \M)$ satisfying Hypothesis \ref{hyp}. The answer to the question of whether $(\mathcal D, \leq_\M)$ is closed under pushouts or has the amalgamation property depends substantially on the particular choice of the class $\mathcal D$ and the class of embeddings $\mathcal M$, as well as on properties of the ring $R$. We will illustrate this by several examples. 

\begin{example}\label{illus1} \rm $(\mathcal F _0, \leq_p)$ and $(\mathcal T \mathcal F, \leq_{RD})$ over any ring $R$ are closed under pushouts and have the amalgamation property.
\end{example}

\begin{example}\label{illus2} \rm  Let $R$ be an integral domain which is not a field. Then $(\mathcal T \mathcal F, \subseteq)$ is not closed under pushouts by Example \ref{ex1}. However, $(\mathcal T \mathcal F, \subseteq)$ has the amalgamation property: an amalgam is obtained as the factor of the pushout by its torsion part. 

Since both $R$ and its quotient field $Q$ are flat modules, Example \ref{ex1} also shows that $(\mathcal F_0, \subseteq)$ is not closed under pushouts. However, $(\mathcal F_0, \subseteq)$ has the amalgamation property: an amalgam is obtained as the injective envelope of the factor of the pushout by its torsion part. This follows from the fact that injective envelopes of torsion-free modules are the flat modules of the form $Q^{(\kappa)}$ for a cardinal $\kappa$.    
\end{example}  

\begin{example}\label{illus3} \rm  $(\mathcal F_0, \subseteq)$ need not have amalgamation in general. We present an example where $R$ is a particular hereditary finite-dimensional algebra over a field, so $\mathcal F _0 = \mathcal P _0$. Observe that $(\mathcal F_0, \subseteq)$ is an AEC.

Let $K$ be a field and $Q$ be the Dynkin quiver $D_4$, i.e., the quiver with four vertices $e, e_0, e_1, e_2$ and three arrows $\alpha : e_0 \to e$, $\alpha_1 : e_1 \to e_0$, and $\alpha_2 : e_2 \to e_0$. Let $R = KQ$ be the path algebra of $Q$. Recall that the Jacobson radical $J$ of $R$ is the ideal generated by the arrows of $Q$. 

$R$ possesses non-isomorphic non-projective simple modules $S_1 = e_1R/e_1J$ and $S_2 = e_2R/e_2J$ where $e_iJ = e_i \alpha_i R \overset{\varphi_i}\cong e_0R$ and $\varphi_i(e_i \alpha_i) = e_0$ ($i = 1, 2$). Using the pushout of the embedding $e_iJ \hookrightarrow e_iR$ and $\varphi_i$, we obtain a projective module $P_i$ and an isomorphism $\Phi_i : e_iR \to P_i$ extending $\varphi_i$ for $i = 1,2$. So $P_i = p_iR$ where $p_i = \Phi_i(e_i)$, and $p_i \alpha _i = \Phi(e_i \alpha_i) = \varphi_i(e_i \alpha_i) = e_0$ ($i = 1,2$). Let $f_i$ denote the inclusion of $e_0R$ into $P_i$. Then the cokernel of $f_i$ is isomorphic to $P_i/e_0R \cong e_iR/e_iJ =  S_i$ ($i = 1, 2$). 

Let $(P, g_1,g_2)$ be the pushout of $(f_1, f_2)$. So $P = (P_1 \oplus P_2)/I$ where $I = (e_0,-e_0)R$, the monomorphisms $g_1$, $g_2$ are defined by $g_1(p_1) = (p_1,0) + I$, $g_2(p_2) = (0,p_2) + I$, and their cokernels are isomorphic to $S_2$ and $S_1$, respectively. The annihilator of the element $0 \neq x = (p_1,p_2) + I \in P$ is $eR \oplus e_0R \oplus (\alpha_1 - \alpha_2)R$. So $xR \cong (e_1R \oplus e_2R)/(\alpha_1 - \alpha_2)R$. As $(\alpha_1 - \alpha_2)R \subseteq e_1J \oplus e_2J$, and the latter module is the Jacobson radical of $e_1R \oplus e_2R$, $(\alpha_1 - \alpha_2)R$ is a non-zero small submodule, hence not a direct summand, in $e_1R \oplus e_2R$. This implies that $xR$ is not projective. Since $R$ is right hereditary, neither $P$ is projective.  Therefore, $(\mathcal F_0, \subseteq)$ is not closed under pushouts.

Assume for the sake of contradiction that there exists an amalgam $(F, h_1, h_2)$ with $F \in \mathcal F _0$ of $(f_1, f_2)$. Then $F$ factors through $P$, i.e., there is a homomorphism $t : P \to F$ such that both $t g_1$ and $t g_1$ are monic. Since $R$ is right hereditary, the image of $t$ is projective, so without loss of generality $t$ is a split epimorphism. However, $t$ is not monic, as $P$ is not projective. As $\operatorname{Ker} t \cap \operatorname{Im}{g_i} = 0$ and $\operatorname{Im}{g_i}$ is a maximal submodule of $P$, we have $\operatorname{Ker} t \oplus \operatorname{Im}{g_i} = P$, whence $\operatorname{Ker} t \cong S_i$ for $i = 1, 2$, in contradiction with $S_1 \ncong S_2$. Therefore, $(\mathcal F_0, \subseteq)$ does not have the amalgamation property.
\end{example}

In the setting of Example \ref{illus3}, $\mathcal F_0 = \mathcal P _0$, so $(\mathcal P _0, \leq_p)$ is closed under pushouts, but $(\mathcal P _0, \subseteq)$ does not have the amalgamation property. A similar construction, but over a different hereditary ring $R$, presented in the following example, shows that it may even happen that $(\mathcal P _0,\leq_\M)$ does not have the amalgamation property in any infinite cardinal $\lambda$ and for any choice of $\M$, even though $(\mathcal P _0,\leq_\M)$ is categorical in all uncountable cardinals.

\begin{example}\label{illus4} \rm Consider a simple countable von Neumann regular ring $R$ which is not artinian and such that all projective modules are free. Such rings can be constructed, for example, using the downward L\"{o}wenheim-Skolem Theorem as subrings of the rings $S/I$, where $S$ is the endomorphism ring of a linear space of dimension $\kappa \geq \aleph_0$, and $I$ is the maximal two-sided ideal of $S$ consisting of all endomorphisms of rank $< \kappa$. Since $R$ is von Neumann regular, all modules are flat, hence all embeddings are RD and pure. As $R$ is countable, each right ideal is countably generated, whence $R$ is right hereditary \cite[Lemma 3.2]{T}. Since $\mathcal P _0 = Free$, $R$ is categorical in all uncountable cardinals. 

Note that $R$ has $2^{\aleph_0}$ non-isomorphic simple modules by \cite[Proposition 6.3]{T}, and they are all non-projective, because $R$ is a simple non-artinian ring. Let $S_1$ and $S_2$ be two non-isomorphic (non-projective) simple modules. Then $R$ has maximal right ideals $I_1$ and $I_2$ such that $S_1 \cong R/I_1$ and $S_2 \cong R/I_2$. Since both $S_1$ and $S_2$ are non-projective, the maximal right ideals $I_1$ and $I_2$ are countably infinitely generated and free, so $I_1 \cong I_2 \cong R^{(\omega)}$. 

Let $f_1: R^{(\omega)} \hookrightarrow R$ and $f_2: R^{(\omega)} \hookrightarrow R$ be the embeddings whose cokernels are $S_1$ and $S_2$, respectively. Let $( P, g_1,g_2)$ denote the pushout of $(f_1, f_2)$. So $g_1: R \to P$ and $g_2: R \to P$ are monomorphisms such that $\operatorname{Coker}{g_1} \cong S_2$ and $\operatorname{Coker}{g_2} \cong S_1$. $P$  is not a projective module, because $P \cong (R \oplus R)/I$ where $I = \{ (f_1(a),-f_2(a)) \mid a \in R^{(\omega)} \} \cong R^{(\omega)}$, so $I$ is infinitely generated, hence not a direct summand in $R \oplus R$. Therefore, $(\mathcal P_0, \leq_\M)$ is not closed under pushouts for any choice of $\M$.

Let $\lambda$ be an infinite cardinal, and $1_\lambda : R^{(\lambda)} \to R^{(\lambda)}$ be the identity homomorphism. Let $f_{i,\lambda} = f_i \oplus 1_\lambda$ ($i = 1, 2$). Then the pushout of $(f_{1,\lambda}, f_{2,\lambda})$ is $(P_\lambda, g_{1,\lambda}, g_{2,\lambda})$ where $P_\lambda = P \oplus R^{(\lambda)}$ and $g_{i,\lambda} = g_i \oplus 1_\lambda$ ($i = 1, 2$). Moreover, $\operatorname{Coker}{g_{1,\lambda}} \cong S_2$ and $\operatorname{Coker}{g_{2,\lambda}} \cong S_2$.      

The rest of the proof is similar to the one given in Example \ref{illus3}: Assume for the sake of contradiction that there exists an amalgam $(F_\lambda, h_{1,\lambda}, h_{2,\lambda})$ of $(f_{1,\lambda}, f_{2,\lambda})$ with $F_\lambda \in \mathcal P _0$. Then $F_\lambda$ factors through the pushout $P_\lambda$, i.e., there is a homomorphism $t_\lambda : P_\lambda \to F_\lambda$ such that both $t g_{1,\lambda}$ and $t g_{2,\lambda}$ are monic. Since $R$ is right hereditary, the image of $t_\lambda$ is a free module, so without loss of generality $t_\lambda$ is a split epimorphism. However, $t_\lambda$ is not monic, because $P_\lambda$ is not projective. As $\operatorname{Ker} t_\lambda \cap \operatorname{Im}{g_{i,\lambda}} = 0$ and $\operatorname{Im}{g_{i,\lambda}}$ is a maximal submodule of $P_\lambda$, we have $\operatorname{Ker} t_\lambda \oplus \operatorname{Im}{g_{i_\lambda}} = P_\lambda$, whence $\operatorname{Ker} t_\lambda \cong S_i$ for each $i = 1, 2$, in contradiction with $S_1 \ncong S_2$. Therefore, $(\mathcal P_0, \leq_\M)$ does not have the amalgamation property in $\lambda$. 
\end{example}

\begin{remark}\label{moreprop} \rm
In Example \ref{illus4}, $R$ is not a right perfect ring, so the class $\mathcal P_0$ is not closed under direct limits. Hence $\mathbf{D} = (\mathcal P_0, \leq_\M)$ \emph{is not an AEC} for any choice of $\M$ by Corollary \ref{corAEC}. In fact, $\mathcal P_0$ is not even closed under $< \lambda$-direct limits for any $\lambda = \aleph_n$ and any $n \geq 0$, and assuming V = L, under $< \lambda$-direct limits for any uncountable regular non-weakly compact cardinal $\lambda$. This follows from the fact (proved in \cite[Theorem VII.1.4]{EM}) that in all these cases there exist non-free $\lambda$-free modules. However, by Corollary \ref{wAEC}, $\mathbf{D}$ is not that far from being an AEC in the sense that only the first of the Tarski-Vaught Axioms, closure of under unions of $\M$-chains, fails.

In Example \ref{illus4}, for any choice of $\M$, $\LS(\mathbf{D}) = \aleph_0$, $\mathbf{D}$ does not have the amalgamation property in $\lambda$ for any infinite cardinal $\lambda$, but $\mathbf{D}$ is categorical in all uncountable cardinals. This contrasts with the following classical result proved \emph{for all AECs} by Shelah in \cite{sh88} (see also \cite[Theorem 17.11]{baldwinbook09}): If $\mathbf{A}$ is an AEC, $\lambda$ is any cardinal such that $\LS(\A) \leq \lambda$, $2^\lambda < 2^{\lambda^+}$, $\mathbf{A}$ is categorical in $\lambda$, and amalgamation fails in $\lambda$, then $\mathbf{A}$ is not categorical in $\lambda^+$; in fact, $\mathbf{A}$ has $2^{\lambda^+}$ non-isomorphic models of cardinality $\lambda^+$.
\end{remark}

\section{Main results}

Throughout this section we fix $(\mathcal{D}, \M)$  such that $(\mathcal{D}, \M)$ satisfies Hypothesis \ref{hyp}. The main result of this section (and the paper) is that the abstract class $\mathbf{D} = ( \md, \leq_\M)$ is stable (Theorem \ref{st}).

 \subsection{Stability} Independence relations on categories were introduced in \cite{lrv1}, these extend Shelah's notion of non-forking which in turn extends linear independence. In this section we study a specific independence relation on $\mathbf{D}$ while in parallel quickly introducing the notions from  \cite{lrv1} in the context of ACs of modules. 

\begin{defin}[{ \cite[3.4]{lrv1}}]\label{ind-rel} \rm 
An \emph{independence relation} on an AC of modules  $\A$ is a collection $\dnf$ of commutative squares of $\A$-embeddings such that for any commutative diagram (with all maps $\A$-embeddings):

\[\begin{tikzcd}
                                         &                                          & A_3 \\
A_1 \arrow[r, "h_1"'] \arrow[rru, "g_1"] & D \arrow[ru, "t" description]            &     \\
A_0 \arrow[u, "f_1"] \arrow[r, "f_2"']   & A_2 \arrow[u, "h_2"] \arrow[ruu, "g_2"'] &    
\end{tikzcd}.\]

we have that $(f_1, f_2, g_1, g_2) \in \dnf$ if and only if $(f_1, f_2, h_1, h_2) \in \dnf$. We call the elements of $\dnf$ \emph{independent squares}.
\end{defin}

We study the following relation on $\D$.

\begin{defin}\label{r-free}  \rm Let $(f_1,f_2)$ be a span such that $f_1: A_0 \to A_1$ and $f_2: A_0 \to A_2$ are $\M$-embeddings, and $A_0, A_1, A_2 \in \md$. Assume  $(A_3,g_1,g_2)$ is an amalgam of $(f_1, f_2)$  such that $g_1: A_1 \to A_3$ and $g_2: A_2 \to A_3$ are $\M$-embeddings and $A_3 \in \md$. Let $(P,h_1,h_2)$ denote the pushout of $(f_1,f_2)$ in $\rmod R$. Then we have the following commutative diagram 

\[\begin{tikzcd}
                                         &                                          & A_3 \\
A_1 \arrow[r, "h_1"'] \arrow[rru, "g_1"] & P \arrow[ru, "t" description]            &     \\
A_0 \arrow[u, "f_1"] \arrow[r, "f_2"']   & A_2 \arrow[u, "h_2"] \arrow[ruu, "g_2"'] &    
\end{tikzcd}.\]

We say that $(f_1, f_2, g_1, g_2) \in \dnf $ in case the (unique) map $t: P \to A_3$ is a $\M$-embedding.  We will often write $\dnf(f_1, f_2, g_1, g_2)$ instead of $(f_1, f_2, g_1, g_2) \in \dnf$.
\end{defin}

\begin{remark} \rm
Definition \ref{r-free} is close to \cite[2.2]{lrvcell}, but the key difference is that we do not assume that $P$ is in $\mathcal{D}$. Of course, we always have $P \in \mathcal{D}$ in the case when $\mathcal{D}$ is closed under $\M$-submodules.

\end{remark}

Observe that in Definition \ref{r-free} $f_1, f_2, g_1, g_2$ are $\D$-embeddings while $t$ is only assumed to be an $\M$-embedding  and $t$ is a $\D$-embedding if and only if  $P \in \mathcal{D}$. Due to this, contrary to the usual convention on papers on AECs, functions in this section are not necessarily $\mathbf{D}$-embeddings and instead we make sure to explicitly mention the nature of each function we introduce.

Notice that by Lemma \ref{test}, $\dnf (f_1, f_2, g_1, g_2)$ always implies that $(A_3,g_1,g_2)$ is a disjoint amalgam of the span $(f_1,f_2)$, and if $\M$ is the class of all embeddings, then $\dnf (f_1, f_2, g_1, g_2)$ is even equivalent to $(A_3,g_1,g_2)$ being a disjoint amalgam of $(f_1,f_2)$. 

However, when $\M$ is the class of all pure or RD-embeddings, $\dnf (f_1, f_2, g_1, g_2)$ is a stronger assumption than $(A_3,g_1,g_2)$ being a disjoint amalgam of the span $(f_1,f_2)$ in general: 

\begin{example}\label{split} \rm Let $R$ be an integral domain with the quotient field $Q$. Assume $R$ is not local, so there exists $r \in R$ such that neither $r$ nor $1 - r$ are invertible in $R$. Let $A_0 = 0$, $A_1 = A_2 = R$, and $f_1 = f_2 = 0$. Let $(P,h_1,h_2)$ denote the pushout of $(f_1,f_2)$ in $\rmod R$. Then $P = R \oplus R$ and $h_i$ embeds $R$ into the $i$th component of $P$ ($i = 1, 2$). Let $A_3 = Rr^{-1} \oplus Rr^{-1} \subseteq Q \oplus Q$, and $g_1, g_2 : R \to A_3$ be defined by $g_1(1) = (r^{-1},1)$ and $g_2(1) = (1,r^{-1})$. 

Since $r^2 \neq 1$, $(A_3,g_1,g_2)$ is a disjoint amalgam of the span $(f_1,f_2)$. Moreover, $t : P \to A_3$ satisfies $t(1,0) = (r^{-1},1)$ and $t(0,1) = (1,r^{-1})$, and $t$ is monic by Lemma \ref{test}. However, as $1 - r$ is not invertible in $R$, $(r^{-1},-r^{-1})$ is not an element of the image $I$ of $t$, whence $t$ is not surjective. So $I$ is a free module of rank $2$ properly embedded in the free module $A_3$ of the same rank. Hence $I$ is an essential submodule of $A_3$, and $t$ is not an RD-embedding. 
\end{example}

Notice that in Example \ref{split}, all the maps $f_i,h_i,g_i$ ($i = 1, 2$) are split embeddings, while the embedding $t$ is not even an RD-embedding (and hence not a pure embedding).     

\medskip
Some basic properties of $\dnf$ can be shown as in \cite[2.7]{lrvcell}. 

\begin{lemma}\label{basic-nf}\
\begin{enumerate}
\item $\dnf$ is an independence relation.
\item (\cite[3.13]{lrv1}) $\dnf$ is transitive, i.e., if \[ \begin{tikzcd}
A_1 \arrow[rr, "g_1"']                   &      & A_3 \arrow[rr, "h_3"']                    &      & A_5                    \\
                                         & \dnf &                                           & \dnf &                        \\
A_0 \arrow[rr, "f_2"'] \arrow[uu, "f_1"] &      & A_2 \arrow[uu, "g_2"'] \arrow[rr, "h_2"'] &      & A_4 \arrow[uu, "h_4"'] 
\end{tikzcd}\]

then $\dnf(f_1, h_2 \circ f_2, g_1 \circ h_3, h_4)$.
\item (\cite[3.9]{lrv1}) $\dnf$ has symmetry, i.e., $\dnf(f_1, f_2, g_1, g_2)$ if and only if $\dnf(f_2, f_1, g_2, g_1)$.
\end{enumerate}
\end{lemma}

Recall that an independence relation $\dnf$ \emph{has existence} \cite[3.10]{lrv1} if every span can be completed to an independent square. As $\mathbf{D}$ might not even have amalgamation (see Examples \ref{illus3} and \ref{illus4}), it is possible that no independence relation on $\mathbf{D}$ has existence. Due to this we introduce a strengthening of $\M$ and show a \emph{weakening} of existence for $\dnf$.

\begin{defin} \rm 
Let $A, B \in \md$. 
 $f: A \to B$ is a \emph{$\mq$-embedding} if $f$ is a $\M$-embedding and $B/f(A) \in \md$. If $f$ is the inclusion from $A$ to $B$, we say that $A$ is a \emph{$\mq$-submodule} of $B$ and denote it by $A \leq_{\mq} B$. 
\end{defin}

\medskip

Our next lemma follows immediately from the commutative diagram presented at the beginning of Section \ref{amalg}. 

\begin{lemma}\label{push} Assume $A_0, A_1, A_2 \in \mathcal{D}$.
If $f_1: A_0 \to A_1$ is a $\mq$-embedding and  $f_2: A_0 \to A_2$ is a $\M$-embedding, then the pushout $( P , g_1,  g_2)$ of $(f_1, f_2)$  (in $\rmod R$) satisfies that $P \in \md$ and $g_2$ is a $\mq$-embedding.  Similarly, if $f_2$ is a $\mq$-embedding, then $g_1$ is a $\mq$-embedding. 
\end{lemma}
\begin{proof} Observe we have the following exact sequence

\[ 0 \xrightarrow[]{} A_2 \xrightarrow[]{g_2}   (A_1 \oplus A_2)/I \xrightarrow[]{h}  A_1/f_1(A_0) \xrightarrow[]{}  0\]

where $I = \{ (f_1(a),-f_2(a)) \mid a \in A_0 \}$ and $h((n_1, n_2) + I)= n_1 + f_1(A_0)$. 

Since $A_2 \in \md$ and  $A_1/f_1(A_0) \in \md$ as $f_1: A_0 \to A_1$ is a $\mq$-embedding, it follows that $ (A_1 \oplus A_2)/I  \in \md$ by closure under extensions (Fact \ref{s-ext}). Moreover, the short exact sequence also shows that $g_2: A_2 \to (A_1 \oplus A_2)/I$ is  a $\mq$-embedding. The final statement can be shown using an analogous argument.  \end{proof}

\begin{cor}(Weak existence) Assume $A_0, A_1, A_2 \in \mathcal{D}$.
If $f_1: A_0 \to A_1$ is a $\mq$-embedding and  $f_2: A_0 \to A_2$ is a $\M$-embedding, then there are $g_1: A_1 \to A_3$ a $\M$-embedding and $g_2: A_2 \to A_3$ a $\mq$-embedding such that $\dnf(f_1, f_2, g_1, g_2)$. Moreover, if $f_2$ is a  a $\mq$-embedding, then $g_1: A_1 \to A_3$ can be chosen to be a $\mq$-embedding. 
\end{cor}
\begin{proof}  Let  $(g_1: A_1 \to P, g_2: A_2 \to P)$ where $(P, g_1, g_2)$ is the pushout of $(f_1, f_2)$. Observe that $\dnf(f_1, f_2, g_1, g_2)$, since $\id_P: P \to P$ is an $\M$-embedding and $P \in \mathcal{D}$ by Lemma \ref{push}. 
\end{proof}

\begin{remark} \rm
In the previous statement, the roles of $f_1$ and $f_2$ can be reversed by symmetry and we may assume that $g_2$ is an inclusion as $\dnf$ is an independence relation.  
\end{remark}

\begin{defin}[{\cite[3.2]{lrv1}}] \rm 
Assume $\mathbf{A}$ is an AC of modules.
We define a relation $\sim_*$ of two amalgamation diagrams of $\mathbf{A}$-embeddings $(f_1: A_0 \to A_1, f_2: A_0 \to A_2, g_1: A_1 \to A_3, g_2: A_2 \to A_3)$ and $(f_1: A_0 \to A_1, f_2: A_0 \to A_2, h_1 : A_1 \to A_3', h_2: A_2 \to A_3')$ as follows: 

$(f_1, f_2, g_1, g_2) \sim_*(f_1, f_2, h_1, h_2)$, if there are $(s_1: A_3' \to T, s_2: A_3 \to T)$ $\mathbf{A}$-embeddings such that the following diagram commutes:

\[\begin{tikzcd}
                                                                 & A_3' \arrow[r, "s_1"]                    & T                     \\
A_1 \arrow[rr, "g_1\hspace{2cm}"', shift left] \arrow[ru, "h_1"] &                                          & A_3 \arrow[u, "s_2"'] \\
A_0 \arrow[u, "f_1"] \arrow[r, "f_2"']                           & A_2 \arrow[ru, "g_2"'] \arrow[uu, "h_2"] &                      
\end{tikzcd}\]

In other words, $(T,s_1,s_2)$ is an amalgam of $(h_1, g_1)$ and  $(h_2, g_2)$. We say that $(s_1, s_2)$ \emph{witness} that $(f_1, f_2, g_1, g_2) \sim_*(f_1, f_2, h_1 , h_2)$. 

Clearly, $\sim_*$ is reflexive and symmetric. We let $\sim$ denote the transitive closure of $\sim^*$.
\end{defin}

\begin{remark}\label{easy} \rm Consider the commutative diagram from Definition \ref{ind-rel}. Then $(t, \id_{A_3})$ witnesses that $(f_1, f_2, g_1, g_2) \sim_*(f_1, f_2, h_1 , h_2)$. 
\end{remark}

It is easy to show that for any AC of modules $\mathbf{A}$: $\dnf$ is an independence relation if and only if $\dnf$ is closed under $\sim$.

We show that \cite[3.12]{lrv1} still goes through with weak existence.

\begin{lemma}\label{b-in} \
\begin{enumerate}
\item If $f_1$ or $f_2$ is an isomorphism then any commutative diagram of $\mathbf{D}$-embeddings of the following form is an independent square:

 \[ \begin{tikzcd}
A_1 \arrow[r, "g_1"]                         & A_3                   \\
A_0 \arrow[u, "f_1"] \arrow[r, "f_2"'] & A_2 \arrow[u, "g_2"']
\end{tikzcd} \]

\item (\cite[3.8]{lrv1}) $\dnf$ has invariance, i.e., $\dnf$ is invariant under isomorphisms of amalgamation diagrams.
\end{enumerate}
\end{lemma}
\begin{proof}\
\begin{enumerate}
\item We may assume without loss of generality that $f_1$ is an isomorphism. Then $f_1$ is a $\mq$-embedding, as $0 \in \md$, so applying weak existence to $(f_1, f_2)$ there are $(h_1, h_2)$ such that $\dnf(f_1, f_2, h_1, h_2)$ with $h_2$ a  $\mq$-embedding. Applying again weak existence to $(g_2, h_2)$ there are $(k_1, k_2)$ such that $\dnf(g_2, h_2, k_1, k_2)$. One can show that $(k_1, k_2)$ witness that $(f_1, f_2, g_1, g_2) \sim_* (f_1, f_2, h_1, h_2)$ using that $f_1$ is an isomorphism. As  $\dnf(f_1, f_2, h_1, h_2)$, hence $\dnf(f_1, f_2, g_1, g_2)$.
\item  The argument of \cite[3.18]{lrv1} goes through by (1) and the fact that $\dnf$ is transitive by Lemma \ref{basic-nf}.
\end{enumerate}
 \end{proof}

Another important property of an independence relation is uniqueness \cite[3.13]{lrv1}. Uniqueness is clear in the particular case when the class $\mathcal D$ is closed under $\mathcal M$-submodules. In general, similarly to existence, we employ a weakening of uniqueness, though we are not aware of any explicit example where uniqueness fails for $\dnf$.

\begin{lemma}\label{wu}(Weak uniqueness)
If $\dnf(f_1, f_2, g_1, g_2)$, $\dnf(f_1, f_2, h_1, h_2)$ and $f_1$ is a $\mq$-embedding, then $(f_1, f_2, g_1, g_2) \sim (f_1, f_2, h_1, h_2)$. More precisely there are $(k_1, k_2)$ $\mathbf{D}$-embeddings such that $(f_1, f_2, g_1, g_2) \sim_* (f_1, f_2, k_1, k_2)  \sim_* (f_1, f_2, h_1, h_2)$.
\end{lemma}
\begin{proof}
Assume $\dnf(f_1, f_2, g_1, g_2)$, $\dnf(f_1, f_2, h_1, h_2)$ and $f_1$ is a $\mq$-embedding. Let $(P, k_1, k_2)$ be the pushout of $(f_1, f_2)$. Then $P \in \md$ by Lemma \ref{push} and there are $d_1: P \to A_3$ and $d_2: P \to A_3'$  $\M$-embedding such that the following diagram commutes:

\[\begin{tikzcd}
                                                                                      &                                                                           & A_3 & A_3' \\
A_1 \arrow[r, "k_1"'] \arrow[rru, "g_1"] \arrow[rrru, "h_1", bend left, shift left=2] & P \arrow[ru, "d_1" description] \arrow[rru, "d_2" description]            &     &      \\
A_0 \arrow[u, "f_1"] \arrow[r, "f_2"']                                                & A_2 \arrow[u, "k_2"] \arrow[ruu, "g_2"'] \arrow[rruu, "h_2"', bend right] &     &     
\end{tikzcd}\]

By Remark \ref{easy}, $( d_1, \id_{A_3})$ witness that $(f_1, f_2, g_1, g_2) \sim_* (f_1, f_2, k_1, k_2)$, and $(d_2, \id_{A_3'})$ witness that($f_1, f_2, k_1, k_2) \sim_* (f_1, f_2, h_1, h_2)$. Hence  $(f_1, f_2, g_1, g_2) \sim (f_1, f_2, h_1, h_2)$.
\end{proof}

The standard argument to obtain monotonicity \cite[3.20]{lrv1} in a stable independence relation uses existence and uniqueness which  $\dnf$ might not have. So we provide a slightly different argument. 

 \begin{lemma}(Monotonicity)
 If
 
 \[ \begin{tikzcd}
A_1 \arrow[rrrr, "{f_{1,4}}"']                       &  &                              &  & A_4                          \\
                                                     &  & \dnf                         &  &                              \\
A_0 \arrow[rr, "{f_{0,2}}"'] \arrow[uu, "{f_{0,1}}"] &  & A_2 \arrow[rr, "{f_{2,3}}"'] &  & A_3 \arrow[uu, "{f_{3,4}}"']
\end{tikzcd}\]

then  $\dnf(f_{0,1},  f_{0,2}, f_{1,4}, f_{3,4} \circ f_{2,3})$
 \end{lemma}
\begin{proof}
 Let $(P, h_1, h_2)$ be the pushout of $(f_{0,1}, f_{0,2})$ and $t: P \to A_4
$ be the map given by the universal mapping property of the pushout applied to $(f_{0,1}, f_{0,2}, f_{1,4}, f_{3,4} \circ f_{2,3})$. We show that $t$ is a $\M$-embedding.

 Let $(Q, d_P, d_{A})$ be the pushout of $(h_2, f_{2,3})$. Then $(Q, d_P \circ h_1 , d_A)$ is the pushout of $(f_{0,1}, f_{2,3} \circ f_{0, 2})$ by transitivity of pushouts. Then $s: Q \to A_4$ given by the universal mapping property of the pushout applied to $(f_{0,1}, f_{2,3}  \circ f_{0,2}, f_{1,4}, f_{3,4})$ is a $\M$-embedding by assumption. One can show that the following diagram commutes:

\[\begin{tikzcd}
                                         &                                          & A_4 \\
A_1 \arrow[r, " h_1"'] \arrow[rru, "f_{1,4}"] & P \arrow[ru, "s \circ d_P" description]  &   \\
A_0 \arrow[u, "f_{0,1}"] \arrow[r, " f_{0,2}"']   & A_2 \arrow[u, "h_2"] \arrow[ruu, "f_{3,4} \circ  f_{2,3}"'] &  
\end{tikzcd}\]

Hence $t= s \circ d_P$ by the universal mapping property of the pushout  $(P, h_1, h_2)$. Observe  $d_P$ is a $\M$-embedding because $f_{2,3}$ is a $\M$-embedding. Therefore, $t$ is a $\M$-embedding.
\end{proof}

\begin{nota}
\rm We write $A_1 \dnf^{A_3}_{A_0} A_2$ if $A_0 \leq_\M A_1, A_2 \leq_\M A_3$, $A_i \in \mathcal{D}$ for $i=0,1,2,3$, and $\dnf(i_{0,1}, i_{0,2},i_{1,3},i_{2,3})$ where $i_{\ell, m}$ is the inclusion map for every $\ell, m$. We read  $A_1 \dnf^{A_3}_{A_0} A_2$ as $A_1$ is \emph{free from} $A_2$ over $A_0$  in $A_3$.
\end{nota}

\begin{lemma}\label{local} (Strong local character)  Assume $A, B, C \in \mathcal{D}$.
If $A, B \leq_\M C$, then there are $ L, T \in \md$ such that $ \| L \|, \| T \| \leq  \| A \| + \kappa$, $L \leq_\M B, T \leq_\M C$, $A \leq_\M T$, $T \leq_{\mq} C$, $L \leq_{\mq} B$ and $ T \dnf^C_{L}  B$.
\end{lemma}

\begin{proof} We prove the case when $\M = \text{Pure}$, the case when $\M = \text{Emb}$ follows from this case and the case when $\M = \text{RD}$ can be obtained modifying the proof as in \cite[3.9]{mj}.

 Let $\mathcal{H}_B$ and $\mathcal{H}_C$ be the Hill families from Fact \ref{Hill} (for the regular cardinal $\kappa^+$ and $\mathcal{C}= \mathcal{D}^{\leq \kappa}$) for $\mathcal{C}$-filtrations $\mathcal{T}_B, \mathcal{T}_C$ of $B$ and $C$ respectively. 

Let $\lambda = \| A \| + \kappa$.
 We construct four increasing chains $\{ K_n \mid n < \omega \} \subseteq \rmod R$, $\{ L_n \mid n < \omega\} \subseteq \md$, $\{ S_n \mid n < \omega \} \subseteq \rmod R$ and $\{ T_n \mid n < \omega \} \subseteq \md$ such that:
\begin{enumerate}
\item $S_0 =A$
\item For every $n < \omega$, $S_n \leq_\M C$, $T_n \in \mathcal{H}_C$ and $S_n \subseteq T_n \subseteq S_{n+1} $
\item  For every $n < \omega$, $K_n \leq_\M B$, $L_n \in \mathcal{H}_B$ and $K_n \subseteq L_n  \subseteq K_{n+1}$
\item For every $n < \omega$, $L_n \subseteq S_{n+1}$
\item For every $n < \omega$, $ \|T_n \|, \| L_n \| \leq \lambda$
\item If $\bar{t} \in T_{n}$, $\varphi(\bar{x}, \bar{y})$ is a $pp$-formula and there is $\bar{b} \in B$ such that $C \vDash \varphi[\bar{t}, \bar{b}]$, then there is $\bar{k}\in K_{n}$ such that $N \vDash \varphi[\bar{t}, \bar{k}]$.

\end{enumerate} 

\fbox{Construction} Fix a well-order of $B^{<\omega}$.

\underline{Base:} Let $S_0 =A$ and $T_0 \in \mathcal{H}_C$ such that $S_0 \subseteq T_0$ and $\| T_0 \| \leq \| S_0 \| + \kappa$ given by Proposition \ref{small}. For each $\bar{t} \in T_0$ and $\varphi(\bar{x}, \bar{y})$ a $pp$-formula, if there is $\bar{b} \in B$ such that $C \models \varphi[\bar{t}, \bar{b}]$ let $\bar{b}_\varphi^{\bar{t}}$ be the least element in $B^{<\omega}$ satisfying this and $\bar{0}$ otherwise. Let $K_{0}$ be the structure obtained by purifying $X= \bigcup \{ \bar{b}_\varphi^{\bar{t}} \mid \bar{t} \in T_0 \text{ and } \varphi \text{ is a }pp\text{-formula} \}$ in $B$ and observe $\| K_0 \| \leq \lambda$.  Let $L_0 \in \mathcal{H}_B$ such that $K_0 \subseteq L_0$ and $\| L_0 \| \leq \| K_0\| + \kappa$ given by Proposition \ref{small}. One can check that everything is  as required.

\underline{Induction step:} Let $S_{n+1}$ be the structured obtained by purifying $L_{n} \cup T_n$ in $C$ and observe $\| S_{n+1} \| \leq \lambda$. Let $T_{n+1} \in \mathcal{H}_C$ such that $S_{n+1} \subseteq T_{n+1}$ and $\| T_{n+1} \| \leq \| S_{n+1} \| + \kappa$ given by Proposition \ref{small}.  Construct $K_{n +1}$ as we constructed $K_{0}$, but replacing $T_{0}$ by $T_{n+1}$ and adding a copy of $L_n$ to $X$. Observe $\| K_n \| \leq \lambda$. Let $L_{n+1}  \in \mathcal{H}_B$ such that $K_{n +1} \subseteq L_{n+1}$ and $\| L_{n+1} \| \leq \| K_{n +1} \| + \kappa$ given by Proposition \ref{small}. One can check that everything is  as required. 

\fbox{Enough} Let $L = \bigcup_{n < \omega} L_n $ and $T = \bigcup_{n< \omega} T_n$. Realize that $L, T \in \md $ by Conditions (2), (3)  and Fact \ref{Hill}.(H2) and  $\| L \|, \| T \| \leq \lambda = \| A \| + \kappa$ by Condition (5). Moreover, $L \leq_\M T, B \leq_\M C$ and $A \leq _\M T$.  Furthermore, $T \leq_{\mq} C$ and $L \leq_{\mq} B$ by Fact \ref{Hill}.(H2),(H3) and Conditions (2), (3). Finally, $T \dnf^C_{L}  B$ by Condition (6) doing an analogous argument to that of \cite[4.14]{maz2}. \end{proof}

The next result uses a similar argument to \cite[12.6]{vasey16}, but  we have weaker assumptions on $\dnf$.

 \begin{prop}\label{claim1} Assume $A_0, A_1, A_1', A_2, A_3, A_3' \in \mathcal{D}$.
  Assume $A_1 \dnf^{A_3}_{A_0} A_2$ and $A_1' \dnf^{A_3'}_{A_0} A_2$ such that $A_0 \leq_{\mq} A_2$. Let $\bar{b}, \bar{b}'$ be enumerations of $A_1, A_1'$ respectively. If $(\bar{b}, A_0, A_1))E_{\text{at}}^\D (\bar{b}', A_0, A_1')$, then $\gtp_\mathbf{D}(\bar{b}/A_2; A_3)=\gtp_\mathbf{D}(\bar{b}'/A_2; A_3')$.

 \end{prop}
\begin{proof} By definition of $)E_{\text{at}}^\D$ there is $g: A_1 \xrightarrow[A_0]{} L$ a $\mathbf{D}$-embedding such that $A_1' \leq_\M L$ and $g(\bar{b}) = \bar{b}'$. Then there is $f: A_3 \cong N$ such that $f \rest_{A_1} = g$ and $g(A_1) \leq_\M N$. Then by invariance applied to $f$ and $A_1 \dnf^{A_3}_{A_0} A_2$, we get that $A_1' \dnf^{N}_{A_0} f(A_2)$. Then by monotonicity the following is an independent square:

\[\begin{tikzcd}
A_1' \arrow[r ]     &  N         \\
A_0 \arrow[u] \arrow[r] & A_2 \arrow[u, "f"]
\end{tikzcd}\]

Then by weak uniqueness (Lemma \ref{wu}) applied to the previous diagram and $A_1' \dnf^{A_3'}_{A_0} A_2$, there are $(k_1: A_1' \to P, k_2: A_2 \to P)$ such that   $(i_{0,1}, i_{0,2}, i_{1,3}, i_{2,3} ) \sim_* (i_{0,1}, i_{0,2}, k_1, k_2)$ via $(h_P: P \to H, h_A: A_3' \to H)$ and $ (i_{0,1}, i_{0,2}, k_1, k_2) \sim_* (i_{0,1}, i_{0,2}, i_{1, N}, f)$ via $(d_N: N \to D, d_P: P \to D)$.  Let $s: P^* \cong P$ such that $A_2 \leq_\M P^*$  and $s \rest A_2 = k_2$.

Observe that the following diagram commutes

\[\begin{tikzcd}
A_3' \arrow[r, "h_A" ]     &  H         \\
A_2 \arrow[u] \arrow[r] & P^* \arrow[u, "h_P \circ s"]
\end{tikzcd}\]

and $h_A(\bar{b}') = h_P \circ s ( s^{-1} \circ k_1 (\bar{b}'))$. Hence  $( \bar{b}', A_2, A_3') )E_{\text{at}}^\D (s^{-1} \circ k_1 (\bar{b}'), A_2, P^*)$. Similarly, it follows that $(s^{-1} \circ k_1 (\bar{b}'), A_2, P^*) )E_{\text{at}}^\D (\bar{b}, A_2, A_3) $ using $(d_P \circ s, d_N \circ f)$. Therefore, $\gtp_{\mathbf{D}}(\bar{b}'/A_2; A_3') =\gtp_{\mathbf{D}}(\bar{b}/A_2; A_3)$. \end{proof}

We obtain our main result. We repeat the assumptions for convenience of the reader.

\begin{theorem}\label{st} Assume $\md$ is $\kappa^+$-deconstructible for $\kappa \geq \operatorname{card}(R) + \aleph_0$ and $\M$ is the collection of either the class of all embeddings, all RD-embeddings or all pure embeddings.

If $\lambda^{\kappa}=\lambda$, then $\D = (\md, \leq_\M)$ is stable in $\lambda$.
\end{theorem}
\begin{proof}
Let $\lambda$ be such that $\lambda^{\kappa}=\lambda$. Let $A \in \md$ with $\| A \| = \lambda$ and assume for the sake of contradiction that $| \gS_\D(A) | > \lambda$. Let  $\{ p_i \mid i < \lambda^+ \}$ be an enumerations without repetitions of types in $\gS(A)$. For every $i < \lambda^+$, let $(b_i, B_i)$ such that $A \leq_\M B_i$, $b_i \in B_i$ and   $p_i = \gtp_\D(b_i/ A; B_i)$.

It follows from Lemma \ref{ls-axiom} and Lemma \ref{local}, that for every $i < \lambda^+$, there are $A_{i} \leq_{\mq} A$ and $B^*_{i} \leq_{\mq} B_i$ such that $\| B_i^* \| = \| A_{i} \| = \kappa$, $b_i \in B^*_{i}$  and $B^*_{i} \dnf^{B_i}_{A_{i}}  A$.

Let $\Phi: \lambda^+ \to [A]^{\kappa }= \{ A' \mid A' \leq_{\M_{\mathcal{Q}}} A \text{ and }  \| A' \|  = \kappa \}$ be given by $\Phi(i)=A_{i}$. Observe that $|[A]^{ \kappa }| = \lambda$ as $\lambda^\kappa =\lambda$. Then by the pigeon-hole principle there is $S \subseteq \lambda^+ $ and $A^* \in [A]^{ \kappa }$ such that $|S| = \lambda^+$ and for every $i \in S$ we have that $\Phi(i)=A^*$.

Let $\{ a^*_\alpha \mid \alpha < \kappa \}$ be  an enumeration of $A^*$. Let $\{c\} \cup \{ d_\alpha \mid \alpha < \kappa \}$ be new constant symbols not in the language of modules and  $ \Delta = \{ (N, c^N, (d_\alpha^N)_{\alpha < \kappa})/ \cong  \mid N \in \md \text{ with } \|N\| = \kappa, c^N \in N \text{ and } d_\alpha^N \in N \text{ for all } \alpha < \kappa \}$.

 Let $\Psi: S \to \Delta$ be given by $\Psi(i)= (B^*_i, b_i, (a^*_\alpha)_{\alpha < \kappa})/\cong$. Since $\| \Delta \| \leq 2^\kappa  \leq \lambda^\kappa  = \lambda$. By the pigeon-hole principle there are $ i \neq j \in S$ and  $f: (B^*_i, b_i, (a^*_\alpha)_{\alpha < \kappa}) \cong (B^*_j, b_j, (a^*_\alpha)_{\alpha < \kappa})$. Observe that $f: B^*_i \to B^*_j$ is an isomorphism, $f(b_i) = b_j$ and $f \upharpoonright A^* = \id_{A^*}$.

Let $\bar{c}$ be an enumeration of  $B^*_i$ starting with $b_i$ and $\bar{c}'= f(\bar{c})$. Observe that $\bar{c}'$ is an enumeration of $B^*_j$ as $f$ is bijective. Furthermore $(\bar{c}, A^*, B^*_i) )E_{\text{at}}^\D (\bar{c}', A^*, B^*_j)$ as $f$ is a witness for it because  $f \upharpoonright A^* = \id_{A^*}$. Since $B^*_{i} \dnf^{B_i}_{A^*}  A$, $B^*_{j} \dnf^{B_j}_{A^*}  A$ and $A^* \leq_{\mq} A$, it follows from Proposition \ref{claim1} that $\gtp_\D(\bar{c}/A, B_i) = \gtp_\D(\bar{c}'/A, B_j)$. Observe that $f(b_i)= b_j$ and $b_i \in \bar{c}$, so it follows that  $\gtp_\D(b_i/A, B_i) = \gtp_\D(b_j/A, B_j)$. This is a contradiction as $p_i \neq p_j$. \end{proof}

The following result was known for $n=0$ and pure embeddings \cite[4.3]{lrvcell} but it is new for embeddings and  RD-embeddings when $n=0$ and for any choice of $\M$ for $n>0$. All these classes are AECs (see Example \ref{q2}).

\begin{cor} 
 $(\mathcal F _n, \leq_\M)$  is stable where $\M$  is the class of all embeddings, all RD-embeddings or all pure embeddings, and $n \in \mathbb{N}$.
\end{cor}

The following result is new. It is worth mentioning that these classes will in most cases not be AECs.

\begin{cor} 
 $(Free, \leq_\M)$ and  $(\mathcal P _{n}, \leq_\M)$ are stable where $\M$  is the class of all embbedings, all RD-embeddings or all pure embeddings, and $n \in \mathbb{N}$. 
\end{cor}

The following result was known for $n=0$ \cite[\S 3]{maz2}, but it is new for $n > 0$. All these classes are AECs (see Example \ref{q2}).

\begin{cor}  Assume $R$ is right noetherian. Then
 $(\mathcal I_n, \leq_\M)$ is stable where $\M$  is the class of all embeddings, all RD-embeddings or all pure embeddings, and $n \in \mathbb{N}$. 
 \end{cor}

 \subsection{Tameness}  Recall that an AC of modules $\mathbf{A}$ is \emph{$\kappa$-tame} for $\kappa \geq \LS( \mathbf{A})$, if for every $A \in \mathcal{A}$ and $p, q \in \gS(A)$, if $p \neq q$ then there is  $B \lesssim_\mathbf{A} A$ such that $p \rest B \neq q \rest B$ and $\| B \| \leq \kappa$. Tameness was first isolated for AECs in \cite{tamenessone} and has played an important role in the study of abstract elementary classes.

 We study tameness for ACs  $\D = (\mathcal{D}, \leq_\M)$ assuming $(\mathcal{D}, \M$)  satisfies Hypothesis \ref{hyp}.

Lemma \ref{claim3}, Proposition \ref{claim2} and Lemma \ref{w-uniq} are similar to \cite[8.5]{lrv1}, but we have weaker assumptions and a weaker conclusion. 

\begin{lemma}\label{claim3} (Local uniqueness for types) Assume that $A_1 \dnf^{A_3}_{A_0} A_2$, $A_0 \leq_{\mq} A_2$, $A_1 \leq_{\mq} A_3$ and $\bar{b}_1, \bar{b}_2 \in A_1^{<\infty}$. If $(\bar{b}_1, A_0, A_3) E_{\text{at}}^\D (\bar{b}_2, A_0, A_3)$, then $\gtp_\D(\bar{b}_1/A_2; A_3)=\gtp_\D(\bar{b}_2/A_2; A_3)$.
\end{lemma} 
\begin{proof} By the definition of $E_{\text{at}}^\D$, there are $f, A_3'$ such that $f(\bar{b}_1) = \bar{b}_2$ and the following diagram commutes:
\[\begin{tikzcd}
A_3 \arrow[r, "f"]      & A_3'          \\
A_0 \arrow[u] \arrow[r] & A_3 \arrow[u]
\end{tikzcd}\]

Since $A_1 \leq_{\mq} A_3$ and $A_1 \leq_\M A_3'$ by weak existence (Lemma \ref{wu}) there are $g$ a $\mq$-embedding and $A_3''$ such that the following is an independent square:

\[\begin{tikzcd}
A_3' \arrow[r, "g"]      & A_3''         \\
A_1\arrow[u] \arrow[r] & A_3 \arrow[u]
\end{tikzcd}\]
 Then  by symmetry, Lemma \ref{b-in} and transitivity applied to the previous independent square, and  by symmetry applied to $A_1 \dnf^{A_3}_{A_0} A_2$, we get:

\[\begin{tikzcd}
A_2 \arrow[rr]            &      & A_3 \arrow[rr]            &      & A_3''                \\
                          & \dnf &                           & \dnf &                      \\
A_0 \arrow[rr] \arrow[uu] &      & A_1 \arrow[uu] \arrow[rr] &      & {g(A_3')} \arrow[uu]
\end{tikzcd}\]

It follows that $A_2 \dnf_{A_0}^{A_3''} g(A_3')$ by transitivity.

Let $h = g\circ f: A_3 \to A_3''$. Since $h(A_1) \leq_\M g(A_3')$ then by monotonicty and symmetry we get that $h(A_1)  \dnf_{A_0}^{A_3''}  A_2$.

Let $\bar{a}$ be an enumeration of $A_1$. Observe that $(\bar{a},A_0, A_1) E_{\text{at}}^\D (h(\bar{a}), A_0, h(A_1))$, then applying Proposition \ref{claim1} we get that  $\gtp_\D(\bar{a}/A_2; A_3) = \gtp_D(h(\bar{a})/A_2; A_3'')$. This is possible by the previous independent square and the fact that $A_1 \dnf^{A_3}_{A_0} A_2$ and $A_0 \leq_{\mq} A_2$.  Finally, observe that $\bar{b}_1 \subseteq \bar{a}$, $h(\bar{b}_1)= \bar{b}_2$ and $\bar{b}_2 \in A_3 \leq_\M A_3''$. Hence $\gtp_\D(\bar{b}_1/A_2; A_3)=\gtp_\D(\bar{b}_2/A_2; A_3)$. 
\end{proof}

We show that $\D$ has some tameness.

\begin{lemma}\label{w-tame}
Let $p_1 =\gtp_\D( b_1/B; B^*)$ and $p_2= \gtp_\D( b_2/B; B^*)$. 

If $(b_1, A, B^*) E_{\text{at}}^\D (b_2, A, B^*)$ for every $A \leq_\M B$ such that $\| A \|  \leq \kappa$, then $p_1 = p_2$.
\end{lemma}
\begin{proof}
It follows from the argument of Lemma \ref{local} applied to $B$ and a module obtained by applying the L\"{o}wenheim-Skolem-Tarski Axiom  \ref{defaec}.(A3) (Lemma \ref{ls-axiom}) to $\{ b_1, b_2 \}$ in $B^*$, that there are  $A \leq_{\mq} B$ and $B_1 \leq_{\mq} B^*$ such that $\| A \| , \| B_1 \| \leq \kappa$, $b_1, b_2 \in B_1$  and $B_1 \dnf^{B^*}_{A}  B$. As $(b_1, A, B^*) E_{\text{at}}^\D (b_2, A, B^*)$ by assumption, we have that $p_1 = p_2$ by Lemma \ref{claim3}.
\end{proof}

Assuming $\mathbf{D}$ has amalgamation, one has that the atomic relation on tuples $E^\D_{at}$ introduced in Definition \ref{gtp-def}.(2) is transitive. Therefore, its transitive closure $E^\D$ is equal to $E^\D_{at}$. Hence it follows directly from the previous result that $\D$ is tame. 

\begin{cor}\label{tam}
If $\mathbf{D} = (\md, \leq_\M)$ has the amalgamation property, then $\mathbf{D}$ is $\kappa$-tame. 
\end{cor}

A natural question is if tameness holds (in ZFC) for $\mathbf{D}$ even without amalgamation. 

\begin{remark} \rm We note that in the particular case when $\mathbf{D}$ has the amalgamation property and it is moreover an AEC, Corollary \ref{tam} together with a recent result by Paolini and Shelah, \cite[1.3]{PS}, yield an alternative proof of the stability of $\mathbf{D}$. \end{remark}

 We finish the paper by showing weak uniqueness for types assuming that $\mathbf{D}$ has amalgamation. 

\begin{prop}\label{claim2}
Assume that $A_2 \dnf^{A_3}_{A_0} A_1$, $A_0 \leq_{\mq} A_2$, $A_1 \leq_{\mq} A_3$ and $A_0 \leq_\M A_1' \leq_\M A_3$. Then there exists $A_1^*, A_3^* \in \mathcal{D}$ and $f: A_1' \xrightarrow[A_0]{} A_1^*$ an $\M$-embedding such that $A_1 \leq_\M A_1^* \leq_\M A_3^*$, $A_3 \leq_\M A_3^*$, $A_1^* \leq_{\mq} A_3^*$ and $A_1^* \dnf^{A_3^*}_{A_0} A_2$.
\end{prop}
\begin{proof}
By weak existence (Lemma \ref{wu}) applied to $A_1 \leq_{\mq} A_3, A_3$ there is $h: A_3 \xrightarrow[A_1]{} N$ a $\mq$-embedding such that $A_3 \leq_{\mq} N$ and $h(A_3) \dnf_{A_1}^N A_3$. Applying $h$ to $A_2 \dnf^{A_3}_{A_0} A_1$, we get that $h(A_2)\dnf^{h(A_3)}_{A_0} A_1$ by invariance.  Then by transitivity applied to the last two independent squares we get that  $h(A_2)\dnf^{N}_{A_0} A_3$.

Let $A_3^* \in \mathcal{D}$ an $g: A_3^* \cong N$ such that $A_3 \leq_{\mq} A_3^*$ and $g \rest A_3 = h$. Let $A_1^* = g^{-1}(A_3)$ and $f := g^{-1}\rest A_1': A_1' \to A_1^*$.  It is straightforward to check that everything holds.
\end{proof}

\begin{lemma}\label{w-uniq}(Weak uniqueness for types) Assume $(\md, \leq_\M)$ has the amalgamation property. Let $p_1 =\gtp_\D( \bar{b}_1/B; B^*)$ and $p_2= \gtp_\D(\bar{b}_2/B; B^*)$  and $A \leq_{\mq} B$. Assume $p_1\rest A = p_2\rest A$ and  $A_1 \dnf^{B^*}_{A} B$, $A_2 \dnf^{B^*}_{A} B$ with $\bar{b}_i \in A_i$, $A_i \leq_{\mq} B^*$ for $i = 1,2$, then $p_1 = p_2$.
\end{lemma}
\begin{proof}
We have that $B \dnf^{B^*}_{A} A_1$ by symmetry. It follows from Proposition \ref{claim2}, that there are $A^*, B^{**} \in \mathcal{D}$ and $f: A_2 \xrightarrow[A]{} A^*$ an $\M$-embedding such that $A_1 \leq_\M A^* \leq_\M B^{**}$, $B^* \leq_\M B^{**}$, $A^* \leq_{\mq} B^{**}$ and $A^* \dnf^{B^{**}}_{A} B$.

Realize that $\bar{b}_1, f(\bar{b}_2)  \in A^*$. Moreover,  $$\gtp_\D(\bar{b}_1/B; B^{**})\rest A = p_1\rest A = p_2\rest A = \gtp_\D(f(\bar{b}_2)/B; B^{**})\rest A$$ where the middle equality holds by hypothesis and the  last equality holds because $f$ fixes $A$. It follows from Lemma \ref{claim3} and amalgamation that $\gtp_\D(\bar{b}_1/B; B^{**})= \gtp_\D(f(\bar{b}_2)/B; B^{**})$.

On the other hand, let $\bar{a}$ be an enumeration of $A_2$ and $f(A_2) = A_2^{*} \leq_\M A^{*}$. Realize $(\bar{a}, A, A_2) E_{\text{at}}^\D (f(\bar{a}), A,  A_2^*)$, $A \leq_{\mq}  B$, $A_2  \dnf_A^{B^*} B$ and $A_2^* \dnf_A^{B^{**}} B$  as $A \leq_\M A_2^* = f(A_2) \leq_M A^*$ and by monotonicity. Then $\gtp_\D(\bar{a}/B; B^{*})=\gtp_\D(f(\bar{a})/A; B^{**})$ by Lemma \ref{claim1}. Observe that $\bar{b}_2 \subseteq \bar{a}$, so $\gtp_\D(\bar{b}_2/B; B^{*})=\gtp_\D(f(\bar{b}_2)/B; B^{**})$. 

Putting together the last equation of the last two paragraph, it follows that $\gtp_\D(\bar{b}_2/B; B^{*}) = \gtp_\D(\bar{b}_1/B; B^{**})$. Hence $p_1 = p_2$  because $B^* \leq_\M B^{**}$.
\end{proof}


\end{document}